 \documentclass[11pt,leqno]{article}
\usepackage{amsthm,amsfonts,amssymb,epsfig,graphics,amsmath,eufrak}
\usepackage[latin1]{inputenc}\relax

\newcommand{\sL}{{L_*^\eps}}

\newcommand{\csL}{{\cL_*^\eps}}
\newcommand{\csLd}{{(\cL_*^\eps)^\dagger}}
\newcommand{\csLe}{{\cL_*^{\eps,\eta}}}

\font\tenronde=rsfs10
\font\sevenronde=rsfs7
\font\fiveronde=rsfs5
\newfam\rondefam
\scriptscriptfont\rondefam=\fiveronde
\textfont\rondefam=\tenronde
\scriptfont\rondefam=\sevenronde
\def\ronde{\fam\rondefam\tenronde}

\newcommand{\sS}{{\ronde S}}

\newcommand{\scL}{{\ronde L}}
\newcommand{\scB}{{\ronde B}}

\def\eps{\varepsilon }
\newcommand{\vp}{\varphi}
\def\blockdiag{\text{\rm blockdiag}}

\renewcommand{\Re}{\text{\rm Re }}

\def\g{\gamma}

\newcommand{\RR}{{\mathbb R}}

\newcommand{\mA}{{\mathbb A}}

\newcommand{\PP}{{\mathbb P}}

\newcommand{\HH}{{\mathbb H}}

\newcommand{\MM}{{\mathbb M}}

\newcommand{\UU}{{\mathbb U}}
\newcommand{\VV}{{\mathbb V}}

\newcommand\cB{{\mathcal  B}} 

\newcommand\cU{{\mathcal  U}}

\newcommand\cV{{\mathcal  V}}

\newcommand\cR{{\mathcal  R}}

\newcommand\cK{{\mathcal  K}}
\newcommand\cL{{\mathcal  L}}
\newcommand\cN{{\mathcal  N}}

\newcommand\cF{{\mathcal  F}}

\newcommand\cT{{\mathcal T}}
\newcommand\cS{{\mathcal S}}

\newcommand{\mez}{{\frac{1}{2}}}

\newcommand{\rmu}{{\mathrm{u} }}

\def\eps{\varepsilon }
\def\blockdiag{\text{\rm blockdiag}}

\def\D{\partial }
\newcommand\adots{\mathinner{\mkern2mu\raise1pt\hbox{.}
\mkern3mu\raise4pt\hbox{.}\mkern1mu\raise7pt\hbox{.}}}

\newcommand{\Id}{{\rm Id }}
\newcommand{\im}{{\rm Im }\, }
\newcommand{\re}{{\rm Re }\, }

\newcommand{\la}{\langle }
\newcommand{\ra}{\rangle }

\newtheorem{theo}{Theorem}[section]
\newtheorem{prop}[theo]{Proposition}
\newtheorem{cor}[theo]{Corollary}
\newtheorem{lem}[theo]{Lemma}

\newtheorem{ass}[theo]{Assumption}

\newtheorem{rem}[theo]{Remark}

\numberwithin{equation}{section}

\newcommand\be{\begin{equation}}
\newcommand\ee{\end{equation}}
\newcommand\bp{\begin{pmatrix}}
\newcommand\ep{\end{pmatrix}}
\newcommand\ba{\begin{aligned}}
\newcommand\ea{\end{aligned} }

\newcommand\bea{\begin{array}}
\newcommand\ena{\end{array}}

\begin{document}

\title{Existence and sharp localization in velocity\\
of small-amplitude Boltzmann shocks}

\author{\sc \small Guy M\'etivier\thanks{
IMB, Universit\'e de Bordeaux, CNRS, IMB, 
33405 Talence Cedex, France; metivier@math.u-bordeaux.fr.:
G.M. thanks Indiana University for its hospitality during a
visit in which this work was partly carried out,
},
Kevin Zumbrun\thanks{Indiana University, Bloomington, IN 47405;
kzumbrun@indiana.edu:
K.Z. thanks the University of Bordeaux I and University of Paris XIII
for their hospitality during visits in which
this work was initiated and partly carried out.
Research of K.Z. was partially supported
under NSF grants number DMS-0070765 and DMS-0300487.
 }}


\maketitle

\begin{abstract}
Using a weighted $H^s$-contraction mapping argument 
based on the macro-micro decomposition of Liu and Yu,
we give an elementary proof of existence,
with sharp rates of decay and distance from the Chapman--Enskog approximation,
of small-amplitude shock profiles of 
the Boltzmann equation with hard-sphere potential,
recovering and slightly sharpening results obtained by Caflisch and Nicolaenko
using different techniques.
A key technical point in both analyses
is that the linearized collision operator $L$ is negative definite
on its range,
not only in the standard square-root Maxwellian weighted norm for which
it is self-adjoint, but also in norms with nearby weights.
Exploring this issue further, we show that $L$ is negative
definite on its range 
in a much wider class of norms including norms with weights
asymptotic nearly to a full Maxwellian rather than its square root.
This yields sharp localization in velocity at near-Maxwellian 
rate, rather than the square-root rate obtained in previous analyses.
\end{abstract}

\tableofcontents


\section{Introduction}\label{intro}

In this paper, we study existence and structure
of small-amplitude shock profiles
\be\label{prof}
f(x,\xi,t)= \bar f(x-st,\xi),
\quad
\lim_{z\to \pm \infty} \bar f(z)=f_\pm
\ee
of the one-dimensional Boltzman equation
\be\label{Boltz}
f_t + \xi_1 \partial_x f= \tau^{-1} Q(f,f),
\ee
$x$, $t\in \RR$, 
where $f(x,t,\xi)\in \RR$ denotes the distribution of
velocities $\xi\in \RR^3$ at point $x$, $t$, 
$\tau>0 $ is the Knudsen number, and
\begin{equation}
\label{colop}
Q (g, h) := \int \big( g( \xi') h (\xi'_*)  -  g(\xi) h(\xi_*) \big) C(\Omega, \xi - \xi_*) d \Omega d\xi_* 
\end{equation}
is the collision operator, with
\begin{equation}\label{Omega}
\begin{aligned}
&\xi \in \RR^3, \quad \xi_* \in \RR^3, \qquad \Omega \in S^2, 
\\
& \xi' =  \xi + \big( \Omega \cdot  (\xi_* - \xi) \big) \Omega 
\\
& \xi'_* =  \xi_* -  \big( \Omega \cdot  (\xi_* - \xi) \big) \Omega.
\end{aligned}
\end{equation} 
and various collision kernels $C$.
Our main example is the hard sphere case, for which 
\begin{equation}\label{hs}
C (\Omega, \xi) = \big| \Omega \cdot \xi \big|. 
\end{equation}
See, e.g., \cite{Gl} for further details.

Note that $Q$ is in this case not symmetric. 
Other standard examples we have in mind are associated with
the class of hard cutoff potentials defined by Grad \cite{G},
as considered in \cite{CN}.
By small-amplitude, we mean that the density 
$$
\rho(x,t):=\langle 1 \rangle_f(x,t):=\int_{\RR^3} f(x,t,\xi)d\xi
$$
is confined within an $\eps_0$-neighborhood of some fixed reference
density $\rho_0>0$ for all $x$, $t$, for $\eps_0>0$ sufficiently small,
where, throughout our analysis, we have fixed
$$
\tau \equiv 1.
$$
Substituting \eqref{prof} into \eqref{Boltz}, we seek, equivalently,
stationary solutions of the {\it traveling-wave equation}
\be\label{Bode}
(\xi_1-s) \partial_x f=  Q(f,f).
\ee
By frame-indifference, we may without loss of generality take $s=0$.

Recall \cite{G,Gl,KMN,CN,LY} that the set of collision invariants
$\langle \psi \rangle $, that is linear forms such that 
$$
 \int_{\RR^3}\psi(\xi)Q(g, g)( \xi)d\xi \equiv 0  , 
$$
is spanned by
\begin{equation}
\label{defR}
Rf:=\langle \Psi\rangle_f  
= \int \Psi (\xi ) f( \xi ) d\xi   \in \RR^5 , \qquad 
\Psi (\xi ) = (1 ,\xi_1 , \xi_2 , \xi_3 , \mez | \xi|^2 )^T.
\end{equation}
Associated with these invariants are the macroscopic
fluid-dynamical variables 
\begin{equation}\label{fvar}
\rmu  := Rf=:
(\rho , \rho v_1 , \rho v_2 ,  \rho v_3 , \rho E )^T,
\end{equation} 
where $\rho$ is density, $v=(v_1,v_2,v_3)$ is velocity,
$E= e +  \mez | v|^2 $ is total energy density, and $e$ is internal energy
density.
Here, we are assuming that $f(x,t,\cdot)$ is confined to a space $\HH$
to be specified later such that the integral converges for  $f \in \HH$. 

Taking moments of \eqref{Boltz} and applying definition \eqref{fvar},
we find that the fluid variables obey the one-dimensional Euler equations 
\be\label{euler}
\ba
\rho_t + \D_x(\rho v_1)=0\\
(\rho v)_t + \D_x(v_1 \rho  v+ pe_1)=0\\
(\rho E)_t + \D_x(v_1 ( \rho E +  p))=0,\\
\ea
\ee
$e_1=(1,0,0)^T$ the first standard basis element,
where the new variable $p=p(f)$, denoting pressure, depends in general
on higher, non-fluid-dynamical moments of $f$.

The set of equilibrium states $Q(f,f)=0$ are exactly (see, e.g., \cite{Gl}) 
the Maxwellians 
\begin{equation}
M_\rmu(\xi)  = \frac{\rho}{\sqrt{(4\pi e/3)^3}}  
e^{-  \frac{| \xi - v |^2}{4 e/3} } . 
\end{equation}
Making the equilibrium assumption $f=M_\rmu$, we obtain a closed
system of equations for the fluid-dynamical variables consisting
of the one-dimensional Euler equations \eqref{euler}
with pressure $p=p(\rho,E)$ given
by the monatomic ideal gas equation of state
\be\label{eos}
p=(2/3)\rho E.
\ee
This corresponds to the zeroth-order approximation obtained
by formal Chapman-Enskog expansion about a Maxwellian state \cite{G,KMN},
where the expansion can be taken equivalently in powers of $\tau$,
or, as pointed out in \cite{L,MaZ1},
in powers of $k$, where $k$ is the frequency in $x$, $t$ of perturbations.
In the present context, it is the latter derivation that is relevant,
since 
(as we shall see better in a moment)
we seek slowly varying solutions near a constant, Maxwellian, state.

The next-, and presumably more accurate,
first-order Chapman-Enskog approximation yields the one-dimensional 
Navier--Stokes equations
\be\label{NS1}
\ba
\rho_t + \D_x(\rho u)=0\\
(\rho v)_t + \D_x(v_1 \rho  v+ p)=(\mu v_{x})_x\\
(\rho E)_t + \D_x(v_1 ( \rho E +  p))=
(\mu v_1 v_x)_x+
(\kappa T_x)_x ,\\
\ea
\ee
where temperature $T$ is related to internal energy by
$e = \frac{3}{2} RT$, $R$ the universal gas constant, and
\be\label{coeffs}
\mu=\mu(T)>0 \quad \hbox{\rm and }\; \kappa=\kappa(T)>0
\ee
are coefficients of viscosity and heat conduction.
In the hard sphere case, 
these may be computed explicitly as
$\mu(T)= (RT)^{1/2}\mu(1/R)$, 
$\kappa(T)= (RT)^{1/2}\kappa(1/R)$ 
 (Chapman's formulae).
For derivations,
see, e.g., \cite{KMN}, Section 3.

By \eqref{euler}, the fluid-dynamical variables associated with
a traveling wave \eqref{prof} must satisfy
\be\label{fode}
\ba
-s \D_x \rho + \D_x(\rho v_1)=0\\
-s\D_x (\rho v) + \D_x(v_1 \rho  v+ pe_1)=0\\
-s\D_x (\rho E) + \D_x(v_1 ( \rho E +  p))=0,\\
\ea
\ee
hence, integrating from $x=-\infty$ to $x=+\infty$,
the {\it Rankine--Hugoniot conditions}
\be\label{RH}
s [\rho] = [\rho v_1],\quad
s [\rho v] = [v_1 \rho  v+ pe_1],\quad
s [\rho E] = [v_1 ( \rho E +  p)],
\ee
where $[h]:=h(f_+)-h(f_-)$ denotes change in $h$ across the shock.

Noting that endstates $f_\pm$ of \eqref{prof} by \eqref{Bode} necessarily
satisfy $Q(f,f)_\pm=0$, we find that they are Maxwellians
$f_\pm=M_{u_\pm}$, and so the associated pressures $p_\pm=p(f_\pm)$ 
are given by the ideal gas formula \eqref{eos},
recovering the standard fact that
endstates of a Boltzmann shock \eqref{prof} are Maxwellians
with fluid-dynamical variables corresponding
to fluid-dynamical shock waves of the Euler equations with monatomic ideal
gas equation of state \cite{G,CN}.

This gives a rigorous if straightforward connection between Boltzmann
shocks and their zeroth order Chapman--Enskog approximation.
The following, main result of this paper gives a rigorous connection
to the first-order Chapman--Enskog approximation given by
the Navier--Stokes equations \eqref{NS1}
in the limit as shock amplitude goes to zero.

Recall \cite{Gi}, for an ideal-gas equation of state \eqref{eos}
under assumptions \eqref{coeffs}, that for each pair of end-states
$u_\pm$ satisfying the Rankine--Hugoniot conditions \eqref{RH},
the Navier--Stokes equations \eqref{NS1} admit a unique up to
translation smooth traveling-wave solution
$$
u(x,t)=\bar u_{NS}(x-st), 
\quad
\lim_{z\to \pm \infty} \bar u_{NS}(z)=u_\pm,
$$
or Navier--Stokes shock.
Moreover, denoting shock amplitude by $\eps:=|u_+-u_-|$,
we have for $\eps>0$ sufficiently small the asymptotic description
\cite{Pe}
\begin{equation}\label{NSbds1}
|\partial_x^k (\bar u_{NS}-u_\pm)|\le C_k \eps^{k+1}e^{-\theta_k \eps|x|},
\quad x\gtrless 0,
\quad
C_k,\,  \theta_k>0, \quad \hbox{\rm all} \; k\ge 0.
\end{equation}

Up to this point in the discussion, we have made essentially no
assumption on the nature of the collision kernel $C(\Omega, \xi)$.
For the analysis of exact profiles, we require specific properties
of $C$.
For simplicity of exposition,
we specialize hereafter to the hard-sphere case \eqref{hs}.
As discussed in Section \ref{genpot}, the arguments extend 
to a more general class of kernels including the hard cutoff potentials of
Grad \cite{G}.
Then, our main result is as follows.

\begin{theo}\label{mainthm}
In the hard-sphere case \eqref{hs},
for any given fluid-dynamical reference state ${u_0}$ and $\eta>0$, 
there exist $\eps_0 > 0$,  $\delta_k > 0$, and $C_k>0$ such that 
for $|u_+-u_0|\le \eps_0$ and $\eps=|u_+-u_-| \le \eps_0$,   
the standing-wave equation \eqref{Bode} has a solution   
$\bar f$ satisfying for all $k\ge 0$
\begin{equation}\label{finalbds}
\begin{aligned}
\big|\partial_x^k (\bar u - \bar u_{NS})(x)\big|
&\le C_k \eps^{k+2}e^{-\delta_k  \eps|x|},\\
\big|\partial_x^k (\bar f- f_{\bar u_{NS}})(x,\xi)\big|
&\le C_k \eps^{k+2}e^{-\delta_k  \eps|x|}
M_{u_0}(\xi)^{1-\eta},\\
\big|\partial_x^k (\bar f- f_{\pm})(x,\xi)\big|
&\le C_k \eps^{k+1}e^{-\delta_k  \eps|x|}
M_{u_0}(\xi)^{1-\eta},\\
\end{aligned}
\end{equation}  
where $\bar u:=R \bar f$ is the associated fluid-dynamical profile.
Moreover, up to translation, this solution is unique
among functions satisfying for $0\le k\le 2$, $c_k$ sufficiently small,
the weaker estimate
\begin{equation}\label{uniquebds}
\big|\partial_x^k (\bar f- f_{\bar u_{NS}})(x,\xi)\big|
\le c_k \eps^{k+1}e^{-\delta_k  \eps|x|}
M_{u_0}(\xi)^{\frac{1}{2}}.
\end{equation}  
\end{theo}

Existence of small-amplitude Boltzmann profiles was established some 
time ago in \cite{CN} for the full class of hard cutoff potentials, 
viewing them as bifurcations from the constant
Maxwellian solution $f\equiv M_{u_-}$, with the somewhat weaker existence
result
$$
\big|\partial_x^k (\bar f- f_{\bar u_{NS}})(x,\xi)\big|
\le C_k \eps^{k+2}
 e^{-\delta_k  \eps|x| - \tau_k  |x|^\beta } 
M_{u_0}(\xi)^{\frac{1}{2}},
$$
$0\le \beta \le 1$,
but also the somewhat stronger result of uniqueness among solutions satisfying
\begin{equation}\label{CNunique}
\big|\bar f- M_{u_{0}}(x,\xi)\big|
\le C \eps e^{-\delta_k  \eps|x| - \tau_k  |x|^\beta } 
M_{u_0}(\xi)^{\frac{1}{2}}
\end{equation}  
for $C>0$ bounded and $\eps>0$ sufficiently small.
For the hard sphere potential,
positivity of profiles, and the improved estimate \eqref{uniquebds}
were shown by Liu and Yu \cite{LY} 
by a ``macro-micro decomposition''
method in which fluid (macroscopic, or equilibrium) and transient
(microscopic) effects are separated and estimated by different techniques.
This was used in \cite{LY} to establish time-evolutionary stability of
profiles with respect to perturbations of zero fluid-dynamical mass,
$\int u(x)dx=0$, and thus,
{\it assuming the existence result of \cite{CN}}, 
to establish positivity of Boltzmann profiles by the positive maximum
principle for the Boltzmann equation \eqref{Boltz}
together with convergence to the Boltzmann profile of its
own Maxwellian approximation: by definition, a perturbation of zero
relative mass in fluid-dynamical variables.

The purpose of the present paper is to obtain existence
from first principles by an elementary argument in the spirit of \cite{LY},
based on approximate Chapman--Enskog expansion combined with
Kawashima type energy estimates \cite{K} (the macro--micro decomposition
of the reference),
but carried out for the {\it stationary} (traveling-wave) rather than the
time-evolutionary equations, and estimating the finite-dimensional
fluid part using sharp ODE estimates in place of the 
sophisticated energy estimates of \cite{LY}.\footnote{
See also \cite{Go,HuZ} in the fluid-dynamical case.}
In this latter part, we are much aided by the more favorable properties
of the stationary fluid equations, a rather standard boundary
value ODE system, as compared to the time-evolutionary equations, a 
hyperbolic--parabolic system of PDE.
This in a sense completes the analysis of \cite{LY}, providing by
a common set of techniques both existence (through the present argument)
and (through the argument of \cite{LY}) positivity.
At the same time it gives a truly elementary 
proof of existence of Boltzmann profiles.

For similar results in the general finite-dimensional relaxation case,
see \cite{MeZ1,MTZ}.
The key new technical observations needed for the infinite-dimensional case
are a way of choosing Kawashima compensators of finite rank 
(see Remark \ref{Kcomp}),
and the fact that the linearized collision operator
remains negative definite on its range 
not only in norms of square-root Maxwellian
weight where it is self-adjoint, but also in norms with nearby weights;
this allows coordinatization with respect to a single global Maxwellian,
avoiding unbounded commutators associated with a changing 
local Maxwellian frame.

In passing, we obtain also the new result of sharp localization in
velocity at near-Maxwellian rate \eqref{finalbds}, which comes
from improved estimates on the linearized collision operator independent
of the basic argument.
A key technical point in all three analyses-- \cite{CN}, \cite{LY},
and the present one--
is that the linearized collision operator $L$ is negative definite
on its range,
not only in the standard square-root Maxwellian weighted norm,
but also in norms with nearby weights.
Exploring this issue further, we show that $L$ is negative
definite on its range 
in a much wider class of norms including norms with weights
asymptotic nearly to a full Maxwellian rather than its square root.
This observation, of interest in its own right, yields through the
same existence argument sharp localization in velocity at near-Maxwellian 
rate, rather than the square-root rate obtained in previous analyses.


Finally, we note that stability of 
small-amplitude Boltzmann shocks has been shown in \cite{LY} 
with respect to small $H^s$ perturbations with zero mass in
fluid variables.
It would be very interesting to continue along the same lines to
obtain a complete nonlinear stability result as in \cite{SX} 
or \cite{MaZ1,MaZ2},
with respect to general, not necessarily zero mass, perturbations.

\section{The nonlinear collision operator}\label{collop}

We begin by a careful study of the collision operator.
Related results may be found, for example, in \cite{C,GPS}.

\subsection{Splitting of the collision operator}\label{splitting}
In view of definition \eqref{colop},  we split
\begin{equation}
\label{splitQ}
Q (g, h)  = Q_+ (g, h) - Q_- (g, h) 
\end{equation}
into gain and loss parts \cite{G,Gl}, where, for
$\Omega$ defined as in \eqref{Omega},
 \begin{equation}
 \label{defQ+}
Q_+ (g, h) (\xi) =  \int Q_\Omega (g, h) d \Omega,
\qquad
   Q_\Omega (g, h)  =  \int  
  g(\xi')  h (\xi'_*)  C (\Omega, \xi - \xi')   d \xi_*
 \end{equation} 
and
\begin{equation}
\label{defQ-}
Q_- (g, h )  = g (\xi)  \nu_h(\xi)  , 
\qquad 
\nu_h (\xi) = \int C (\Omega, \xi_* - \xi) h (\xi_*) d \xi_* d \Omega .
\end{equation}


\subsection{Estimates for $Q_-$} 

In the hard sphere case \eqref{hs},
\begin{equation}
\nu_h (\xi) = \int C (\Omega, \xi_* - \xi) h (\xi_*) d \xi_*  d \Omega 
=  c  \int | \xi  - \eta)\vert  \  h (\eta) d \eta.
\end{equation} 
Here and elsewhere, denote $\la \xi \ra:=(1+|\xi|^2)^{1/2}$
following standard convention.

\begin{lem} 
In the hard-sphere case \eqref{hs},
for $h \ge 0$ with $\la \xi \ra h \in L^1$,  $\nu_h  $ 
is positive, continuous and 
\begin{equation}\label{-wt}
 \la \xi \ra \lesssim   \nu_h (\xi )   
\le  \la \xi \ra \big\| \la \eta \ra h  \big\|_{L^1}. 
\end{equation}
\end{lem} 

\begin{proof}
Evidently,
    $$
    \nu_h (\xi) \le A |\xi |  +  B , \qquad  A = \| h \|_{L^1}, \quad
    B = \| |  \eta |  h \|_{L^1}. 
    $$
    This implies the upper bound.  Next, 
    $$
    \nu_h (\xi) \ge A | \xi |  - B 
    $$
    which implies the lower bound for $| \xi | >  (B +1) / 2A$. For $\xi$ bounded, the 
    integral is continuous and bounded from below. 
\end{proof} 


 \subsection{Estimates for $Q_+$}
  
  Consider the Maxwellians
  \begin{equation}
  \label{refwei}
  \omega_s (\xi ) = e^{ - s | \xi |^2} 
  \end{equation}
  and the weighted $L^2$ spaces 
  $ \HH^s  = \omega_s L^2(\RR^3)$  with norm
  \begin{equation}
 \big\|   f \big\|_{\HH^s}  = \Big(\int e^{ 2 s | \xi |^2} | f(\xi) |^2 d \xi \Big)^\mez
  \end{equation} 
  We use similar definitions for spaces $\HH^s_1$ [resp.  
  $\HH^s_2$ ] of functions of one variable $\xi_1$ [resp. two variables 
  $(\xi_2, \xi_3)$].

  \begin{prop}
  \label{propestQ+1} 
In the hard-sphere case \eqref{hs},
for any $s> 0$,
  \begin{equation}\label{wt}
  \big\| \la \xi \ra^{- \mez}  Q_+(g, h) \big\|_{\HH^s}  \lesssim 
  \big\|   g  \big\|_{\HH^s}  \big\|  \la \xi \ra^\mez  h \big\|_{\HH^s} .
  \end{equation}
    \begin{equation}\label{nowt}
  \big\|   Q_+(g, h) \big\|_{\HH^s}  \lesssim 
  \big\|   g \big\|_{\HH^s}  \big\|  \la \xi \ra h \big\|_{\HH^s} .
  \end{equation}
  \end{prop}

  We first estimate $Q_\Omega$ (see definition \eqref{defQ+})  for a fixed $\Omega \in S^2$. 
    \begin{lem}\label{Q1lem}
In the hard-sphere case \eqref{hs},
  for  $ s > 0$, 
  \begin{equation}
   \big\|  \la \xi \ra^{- \mez} Q_\Omega (g, h)  \big\|_{\HH^s  }  \le  C_s 
    \big\| g \big\|_{\HH^s}  \big\|  \la \xi \ra^\mez
  h \big\|_{\HH^s }
  \end{equation}
  and 
   \begin{equation}
   \big\|  Q_\Omega(g, h)  \big\|_{\HH^s  }  \le  C_s 
    \big\| g \big\|_{\HH^s}  \big\|  \la \xi \ra 
  h \big\|_{\HH^s }
  \end{equation}
  \end{lem} 
  
  \begin{proof}  Choose the 
  orthonormal basis 
  such that 
  $$
  \Omega = (1, 0, 0). 
  $$
  In this case 
  $$
  \xi' = (\xi_{* 1},  \xi_2, \xi_3) , \quad \xi'_* = ( \xi_1, \xi_{* 2}, \xi_{*3}) , 
  \quad  C (\Omega, \xi_* - \xi) = | \xi_1 - \xi_{*1}|. 
  $$
  and $Q_\Omega$ has the more explicit form 
  \begin{equation}
  Q_1 (g, h) (\xi) = 
  \int   g (\eta_1, \xi_2, \xi_3)   h (\xi_1, \eta_2, \eta_3) 
  | \xi_1 - \eta_1|  d \eta 
  \end{equation}
  In particular
  \begin{equation}
  Q_1 (g, h) (\xi)  =  I_g (\xi_1, \xi_2, \xi_3 )  J_h (\xi_1)
  \end{equation}
  with 
  \begin{eqnarray} 
  I_g (\xi_1, \xi_2, \xi_3) & = &\int   | \xi_1 - \eta_1|   \check g (\eta_1, \xi_2, \xi_3) d \eta_1 
  \\ 
  J_h ( \xi_1) &=& \int \check h (\xi_1, \eta_2, \eta_3) d\eta_2 d \eta_3. 
  \end{eqnarray}
  
Note that 
  $ 
  J(\xi_1)  =  \la \xi_1  \ra^{- \mez} 
 e^{ -   s | \xi_1|^2}  j (\xi_1) $ with 
 $$
 j (\xi_1) = \la  \xi_1 \ra^\mez
 \int e^{ - s (\eta_2^2 + \eta_3^2)} (1 + \xi_1^2 + \eta_2^2 + \eta_3^2) ^{ - \frac{1}{4}} H(\xi_1, \eta_2, \eta_3) d\eta_2 d \eta_3
 $$
 with 
 $ \| H \|_{L^2} = \| \la \xi \ra ^\mez h \|_{\HH^s}$. Thus, 
 $$ 
  \| j \|_{ L^2} \lesssim  \big\|  \la \xi \ra^\mez
  h \big\|_{\HH^s }.
  $$
Therefore,  
  $$
    \la \xi \ra^{-\mez} e^{  s | \xi |^2} Q_1 (g, h) (\xi) \le q(\xi) :=  j (\xi_1)  G(\xi) 
  $$
with 
  $$
G(\xi) =    \int   k (\xi_1, \eta_1) \tilde g  (\eta_1, \xi_2, \xi_3) d \eta_1 , 
\qquad 
k(\xi_1, \eta_1) :=  \frac{ | \xi_1- \eta_1| e^{ - s | \eta_1|^2}}{\la \xi_1 \ra } 
  $$
and $\| \tilde g \|_{L^2} \le \| g \|_{\HH^s}$.  Note that 
$$
\sup_{\xi_1}  \big\| k (\xi_1, \cdot ) \big\|_{L^2}  = C < \infty. 
$$
Thus 
$$
G (\xi_1, \xi_2, \xi_3)^2   \le C^2   \phi  (\xi_2, \xi_3) , \qquad \big\| \phi \big\|_{L^1} \le 
\big\| \tilde g \big\|^2_{L^2}
$$
It follows that 
$$
\big\| q \big\|_{L^2}  \le  C  \big\| j  \big\|_{L^2}  \big\| \tilde g \big\|_{L^2} . 
$$
The proof of the second estimate is similar
  \end{proof}

  Integrating over $\Omega $ on the unit sphere, implies Proposition~\ref{propestQ+1}.
  Combining the estimates for $Q_-$ and $Q_+$ implies:

  \begin{cor}\label{Qbd}
In the hard-sphere case \eqref{hs},
  for all $s$, there is  $C_s$ such that 
  \begin{equation}
  \big\| \la \xi \ra^{- \mez} Q (g, h) \big\|_{\HH^s}  \le  C_s 
  \big\| \la \xi \ra^{\mez} g \big\|_{\HH^s}   \big\| \la \xi \ra^{\mez} h \big\|_{\HH^s}  
  \end{equation}
  \end{cor}
  

  \subsection{Further estimates}
  
  \begin{prop}
In the hard sphere case \eqref{hs},
  suppose that $0 < s < s' $, and $g \in \HH^s$ and $h \in \HH^{s'}$. 
  Then 
  \begin{equation}
  \big\| Q_+ (g, h) \big\|_{\HH^s}  \le  C_{s, s'}  \big\| g  \big\|_{\HH^s}  
   \big\| h  \big\|_{\HH^{s'}} .   
  \end{equation}
  
  \end{prop}

  \begin{proof}   
  Introduce 
  $  G = e^{  s    | \xi |^2 } |   g (\xi ) | $ and 
    $H  = e^{   s' | \xi |^2 } |   h (\xi ) | $.  
  Then, when $\Omega = (1, 0, 0)^T$,  
  $$
  q_1 :=      e^{ 2 s | \xi |^2 } | Q_\Omega  (g, h)  (\xi)| 
  $$
  satisfies 
  $$
  q_1 (\xi) \le e^{ (s - s') | \xi_1|^2}  \int  G (\eta_1, \xi_2, \xi_3) H(\xi_1, \eta_2, \eta_3) 
  | \xi_1 - \eta_1|  e^{ - s | \eta |^2 } d \eta  .   
  $$

  By the Cauchy--Schwarz inequality, 
  $$
  | q_1 (\xi ) |^2 \lesssim  \la \xi_1 \ra^2 e^{ 2 (s - s') | \xi_1|^2}  
  \int | G (\eta_1, \xi_2, \xi_3) |^2   \ | H(\xi_1, \eta_2, \eta_3)|^2 d \eta. 
  $$
  where we have used that 
  $$
  \int   | \xi_1 - \eta_1|^2  e^{ - 2 s | \eta |^2 } d \eta \lesssim 
  \la \xi_1 \ra^2
  $$
  Since $s ' > s$, $  \la \xi_1 \ra^2 e^{ 2 (s - s') | \xi_1|^2} $ is bounded 
  and 
  $$
  \| q_1 \|_{L^2} \le \| G \|_{L^2} \| H \|_{L^2}. 
  $$
  Integrating over $\Omega$  implies the result. 
  \end{proof}

     \begin{prop}
In the hard-sphere case \eqref{hs},
  suppose that $0 < s  < s' $, and $g \in \HH^{s'}$ and $h \in \HH^{s}$. 
  Then 
  \begin{equation}
  \big\| Q_+ (g, h) \big\|_{\HH^s}  \le  C_{s, s'}  \big\| g  \big\|_{\HH^{s'}}  
   \big\| h  \big\|_{\HH^{s}} .   
  \end{equation}
  
  \end{prop}

  \begin{proof}   
  Introduce 
  $  G = e^{    s' | \xi |^2 } |   g (\xi ) | $ and 
    $H  = e^{   s | \xi |^2 } |   h (\xi ) | $.  
  Then  
  $$
  q_\Omega  :=      e^{ 2 s | \xi |^2 } | Q_\Omega  (g, h)  (\xi)| 
  $$
  satisfies 
    $$
  q_\Omega  (\xi) \le e^{ (s - s') ( | \xi|^2 - (\xi \cdot \Omega)^2)}  \int  \Phi_\Omega (\xi, \eta)  
  | (\xi  - \eta) \cdot \Omega |  e^{ - s | \eta |^2 } d \eta    
  $$
where  for all $\Omega \in S^2$: 
$$
\big\| \Phi_\Omega \big\| _{L^2 (\RR^3 \times \RR^3) }  
\le  \big\| G \big\|_{L^2}   \big\| H \big\|_{L^2} 
$$
  Integrate over $\Omega$  and use the Cauchy--Schwarz inequality to get 
  $$
\big| e^{ 2s | \xi |^2} Q_+ (g, h) (\xi ) \big|^2 \le   \Big| \int q_\Omega (\xi) d \Omega \Big|^2 
  \le  A  \int | \Phi_\Omega (\xi, \eta) | ^2  d \eta d \Omega 
  $$ 
  with 
  $$
  A := \int  e^{ 2 (s - s') ( | \xi|^2 - (\xi \cdot \Omega)|^2)} 
    | (\xi  - \eta) \cdot \Omega |^2 e^{ - s | \eta |^2 } d \eta     d \Omega.
  $$
  Thus 
  $$
  A \lesssim     1  +   \int  e^{ 2 (s - s') ( | \xi|^2 - (\xi \cdot \Omega)|^2)} 
    |  \xi \cdot \Omega |^2     d \Omega.
  $$
   To compute the integral (for large $\xi$),  we can choose coordinates 
   such that 
   $\xi = (0, 0, t) $, $t > 0$, and parametrize the sphere with angular coordinates
   $\theta \in [0, 2 \pi]$ and $\vp \in [ - \mez \pi, \mez \pi]$, so that 
   $\Omega = (\cos \vp \cos \theta, \cos \vp \sin \theta, \sin\vp)$
   In this case the integral becomes 
   $$
  2 \pi  \int_{- \mez \pi}^{\mez \pi}  e^{ 2 (s - s') t^2 \cos^2 \vp}  t^2 \sin^2 \vp  \cos \vp \  d \vp    
   $$
   which is smaller than
      $$
      \begin{aligned}
  2 \pi  \int_{0}^{\mez \pi} & e^{ 2 (s - s') t^2 \cos^2 \vp}  2 t^2 \sin \vp  \cos \vp \  d \vp   
  \\
 &   2 \pi  \int_{0}^{1}  e^{ 2 (s - s') t^2 u  }   t^2  du  = 
     2 \pi  \int_{0}^{t^2}  e^{ 2 (s - s')  u  }     du  \le  \frac{2\pi}{s' - s}. 
  \end{aligned}
   $$
   Therefore $A$ is uniformly bounded and 
     $$
\int \big| e^{ 2s | \xi |^2} Q_+ (g, h) (\xi ) \big|^2 d \xi  
  \lesssim     \int | \Phi_\Omega (\xi, \eta) | ^2  d \eta d \Omega d \xi 
  \lesssim  \big\| G \big\|^2_{L^2}   \big\| H \big\|^2_{L^2} .
  $$
  \end{proof}

 \begin{cor}
 \label{factionQ}
In the hard-sphere case \eqref{hs},
  suppose that $0 < s < s' $.  and $f \in \HH^s$.
  Then for $a \in \HH^{s'}$ the mappings 
  $f \mapsto Q _- (a, f) $,  $f \mapsto Q _+(a, f) $ and 
  $f \mapsto Q _+(f , a ) $ are bounded from 
  $\HH^s$ to $\HH^s$, with norm controlled by a constant times 
  $\| a \|_{\HH^{s'}}$. 
 \end{cor}
\bigbreak

\begin{rem}
\label{remgenmaxw}  
\textup{The estimates above were proved for convenience for the Gaussian weights
$\omega = e^{  | \xi |^2}$ and $\omega^s$. They immediately extend to 
any Maxwellian weight  $M_{\rmu}$  and $M_\rmu^s$. }
\end{rem}

%
%
%

\section{The linearized collision operator}\label{lincollop}
We next study the linearized collision operator about a Maxwellian
or nearby velocity distribution.
Fix a reference state $\underline \rmu$. The associated Maxwellian 
$M_{\underline \rmu} $ is denoted by $\underline M$.

For $s \in ]0 , 1 ]$,  let 
  $\HH^s$  denote the space of functions $f $ on $\RR^3$  such that 
 \begin{equation}
 \label{norms}
\big\|  f \big\|^2_{\HH^s} =   \int \underline M(\xi)^{- 2 s} | f(\xi)|^2 d\xi < + \infty. 
 \end{equation}
 Note that $\underline M \in \HH^s$ for all $s < 1$. 
 The space $\HH^\mez$  plays a particular role as it will be clear below. 

\begin{prop}
The quadratic mapping $f \mapsto Q (f, f)$ is continuous from 
$\la \xi \ra^{- \mez} \HH^s$ to $\ker R \cap (\la \xi \ra^{\mez}\HH^s)$, 
where $R$ is the operator \eqref{defR} defining the thermodynamical  variables. 
. 
\end{prop} 

\begin{proof} 
The action from $\la \xi \ra^{- \mez} \HH^s $ to $ \la \xi \ra^{\mez}\HH^s$  is 
a consequence of  Corollary~ \ref{Qbd} and Remark~\ref{remgenmaxw}. 
That the image is contained in 
$\ker R$ follows from the known properties of the collision operator: 
\begin{equation}
\forall f \in \HH^s : \qquad R Q (f, f) = 0. 
\end{equation}
\end{proof}

Given a    function  $a $,   the linearized collision operator  at $a$ is
  \begin{equation}
 L_a  g  = Q (a , g) + Q (g, a).
\end{equation}
In particular, we consider first the linearized operator at $a = \underline M$: 
  \begin{equation}
 \underline L  g  = Q'_{\underline M} g = Q (\underline M , g) + Q (g, \underline M).
\end{equation}
Corollary~\ref{factionQ} implies the following result: 

\begin{lem}
\label{lemaction uL}
For all $ s \in [\mez, 1[$, $\underline L$ is a bounded linear operator
from $\la \xi \ra^{- \mez} \HH^s$ to  $\ker R \cap \la \xi \ra^{\mez } \HH^s$ 
and from $\la \xi \ra^{- 1} \HH^s$ to  $ \ker R \cap  \HH^s$. 
\end{lem}


 \subsection{Symmetry and coercivity on  $\HH^{\mez}$. }

Let $\VV = \ker R \cap \HH^\mez$ and let 
 $\UU$ denote the orthogonal complement of $\VV$ in $\HH$. It has dimension 
$5$. 
Noting that 
\begin{equation}
\big(  \chi, f \big)_{L^2}   = \big( \chi \underline M , f \big) _{\HH}
\end{equation}
we see that $\UU$ is spanned by the functions $\psi_j \underline M$. 
An orthogonal   basis  is   
\begin{equation}
\label{orthobasis} 
\phi_j (\xi)  = \chi_j (\xi) \underline M (\xi) , \qquad j= 0, \ldots, 4, 
\end{equation}
with 
\begin{equation}
\begin{aligned}
&\chi_0 (\xi  ) = 1 ,  \qquad \chi_j(\xi) = \frac{\xi_j - \underline v_j}{\sqrt \gamma \underline T} \quad \mathrm{for} \ j= 1,2,3, 
\\
& \chi_4 (\xi ) =  \frac{1}{\sqrt 6} \Big( \frac{|\xi - \underline  v|^2}{\gamma \underline T} - 3 \Big). 
\end{aligned}
\end{equation} 

We denote by $\PP_\UU$  and $\PP_\VV$ the orthogonal projection from 
$\HH^\mez$ to $\UU$ and $\VV$ respectively. In the language of \cite{LY}, $\UU$ is the macroscopic part of $f$
and $\VV$ is the microscopic part.

Note that  $\UU \subset \la \xi \ra^{- \mez} \HH^{\mez} $ 
and  $\UU \subset \HH^s$ for all $s < 1$. Therefore

\begin{lem}
\label{actionPP}
For $s \in [ \mez, 1 [$,    $\PP_V$ maps 
$\HH^s$  onto $\VV^s := \VV \cap \HH^s = \ker R \cap \HH^s$ and 
$\la \xi \ra^\mez \HH^s$ onto $\VV \cap (\la \xi \ra^\mez \HH^s)$. 
\end{lem}

\begin{rem}
\label{remcomm}
\textup{The projections $\PP_\UU$ and $\PP_\VV$ do not commute with 
the operator of multiplication by $\xi_1$.  They are not orthogonal in $\HH^s$ for $s > \mez$, but still produce
a continuous decomposition $\HH^s = \UU \oplus \VV^s$. }

\end{rem}

\begin{prop}
  $\underline L $ is (formally) self adjoint and nonpositive in $\HH^\mez$  
 and definite negative on 
 $\VV$. More precisely,
 
 i) for all $f$ and $g$ in $f \in \la \xi \ra^{-1} \HH^\mez$, 
 \be
 \big( \underline L f, g \big)_{\HH^\mez} =   \big( f,  \underline L g \big)_{\HH^\mez} .
 \ee
 
 ii)   there is $\delta > 0$ such that for all $f \in \la \xi \ra^{-1} \HH$: 
 \begin{equation}
 \label{coercivmez}
 \delta \big\| \la \xi \ra^\mez \PP_{\VV}  f \big\|^2_{\HH^\mez} \le  - \re  \big(  \underline L f, f \big)_{\HH^\mez}.  
 \end{equation}
 \end{prop} 
 
 \begin{proof}[Notes on the proof] This is a classical result in the theory of 
 Boltzmann equation in the  hard sphere case and more generally in the case 
 of hard cut off potentials (see e.g. \cite{C,G, Gl, CN}) 
 
1.  The analysis of section~\ref{collop} splits $\underline L$ into 
 $\underline L= -\nu(\xi)+K$,  
 with 
 $$
K g (\xi)  = -     \int k_1(\xi, \xi_*  ) g(\xi_*) d \xi_*   
 + \int k _2 (\xi , \Omega, \xi_*)  g(\xi') d \xi_* d \Omega 
  + \int k _3 (\xi , \Omega, \xi_*)  g(\xi'_*) d \xi_* d \Omega 
$$  
  with 
\begin{eqnarray*}
&&k_1 (\xi, \xi_*) =  \underline M (\xi)  \int C (\Omega, \xi_* - \xi) d \Omega   =    \underline M  c_0( \xi - \xi_*)
\\
&&k_2 (\xi, \Omega, \xi_*)    =  \underline M (\xi'_*) C (\Omega, \xi_* - \xi) 
\\
&& k_3 (\xi, \Omega, \xi_*) =  \underline M (\xi') C (\Omega, \xi_* - \xi ) 
\end{eqnarray*}
Using 
the conservations 
  \begin{equation}
  \xi + \xi_* = \xi' + \xi'_* , \qquad | \xi |^2 + | \xi_*|^2 = | \xi' |^2 +  | \xi'_*|^2, 
  \qquad d \xi d \xi_* = d \xi' d \xi'_*, 
  \end{equation}
which imply  that  
  \begin{equation}
  \label{relwei}
  \underline M (\xi) \underline M (\xi_*) =  \underline M (\xi') \underline M (\xi'_*), 
  \end{equation} 
 one shows that 
$$
\big( \underline M^{-1} K_j  g, h \big)_{L^2} 
  \big(g, \underline M^{-1} K_j h   \big)_{L^2}
$$
for $j = 1, 2, 3$, implying the symmetry of $\underline L$ in $\HH^\mez$. 

2. One can also argue as follows. By Boltzmann's $H$-theorem,
$$
\int Q(f,f)\log f d\xi \le 0 
$$
for all $f $ withe enough decay at infinity. 
Hence, Taylor expanding about the Maxwellian $\underline M$,
a minimizer of $\int Q(f,f)\log f d\xi $, we obtain symmetry and nonnegativity
of the Hessian,
$$
\int \frac{(\underline L h) h}{\underline M}d\xi\le 0,
$$
giving nonnegativity of $\underline L$ on $\HH^\mez$ and also 
formal self-adjointness. 

3.  It is known that is  $K$ compact in  $\HH^\mez$ and that 
$\ker \underline L = \UU$ (this can be proved   using the formulas above). 
By self-adjointness of $\underline L$ on $\HH^\mez $, to establish
strict negativity on $\VV$, it is sufficient to establish a spectral
gap between the eigenvalue zero and the essential spectrum of 
$\underline L$.   
But, this follows from Weyl's Lemma by comparison of
$\underline L=-\nu+K$ with the multiplication operator by 
$-\nu(\xi) \le  - c_0 < 0$.
\end{proof}

In a more explicit form, the inequality \eqref{coercivmez}  reads
\begin{equation}
\label{coerc16} 
\delta \int  \la \xi \ra  \underline M(\xi) ^{-1} |  \PP_{\VV}  f (\xi )| ^2  d \xi 
\le  - \re \int \underline M(\xi)^{-1}  \underline L f (\xi)   f (\xi)  d\xi . 
\end{equation}

We also point out the following properties which are freely used below and 
which follow from the symmetry of $\underline L$ in $\HH^\mez$ : 
\begin{equation}
\underline L = \PP_\VV \underline L = \underline L \PP_\VV, 
\qquad 
 \PP_\UU \underline L = \underline L \PP_\UU = 0. 
\end{equation}

 
\subsection{Coercivity on $\HH^s$}
 \label{genco}


With $\lambda \ge 0$ to be determined later on, introduce the equivalent norms: 
\begin{equation}
\label{tildenorms}
\big\| f \big\|_{\widetilde \HH^s}^2 :=   
 \big\| f \big\|_{\HH^s}  +   \lambda \big\|   f \big\|^2_{\HH^\mez} . 
\end{equation}

 \begin{prop}
 \label{propcoercs}  For $\mez \le s < 1$, 
 the operator $\underline L$ is continuous from 
 $\la \xi \ra^{-\mez } \HH^s$ to $(\la \xi \ra^{\mez} \HH^s )\cap \VV$
 and  from $\la \xi \ra^{-1 } \HH^s$ to $ \VV^s$, and formally coercive on 
 $\VV^s$ for the norm \eqref{tildenorms}.  
More precisely,  there are  $\lambda \ge 0 $   and  $\delta > 0$ such that 
for all $f \in \la \xi \ra^{-1 } \VV^s$: 
 \begin{equation}
 \label{coercivs}
\delta \big\| \la \xi \ra^\mez\PP_\VV f \big\|_{\widetilde \HH^s}^2  \le  - \re  \big(  \underline L f, f \big)_{\widetilde \HH^s}  \end{equation}
 \end{prop} 
 
 \begin{proof} We want to prove that
 \begin{equation}
\label{coerc18} 
\begin{aligned}
\delta& \int  \la \xi \ra  \big( \underline  M(\xi) ^{-2s} + 
\lambda \underline M(\xi)^{-1} \big) |     f (\xi )| ^2  d \xi 
\\
& \le  - \re \int \big( \underline  M(\xi) ^{-2s} + 
\lambda \underline M(\xi)^{-1} \big) \underline L f (\xi)   f (\xi)  d\xi ,  
\end{aligned}
 \end{equation}
 Following the analysis of Section~\ref{collop}, 
 $$
 \underline L = - \nu_0 \Id + K 
 $$
 where $\nu_0(\xi) \approx \la \xi \ra$ and 
 $K $ is bounded from $\HH^s$ to $\HH^s$, since the Maxwellian $\underline M \in \HH^{s'}$
 for $s < s' < 1$. 
 Hence there is $\delta_1 > 0$ such that 
 $$
 \delta_1 \int  \la \xi \ra  \underline M(\xi) ^{-2s} |     f (\xi )| ^2  d \xi 
\le  - \re \int \underline M(\xi)^{-2s}  \underline L f (\xi)   f (\xi)  d\xi + C \int 
  \underline M(\xi) ^{-2s} |     f (\xi )| ^2  d \xi . 
$$
Moreover, by \eqref{coerc16}, there is $\delta_0$ such that 
$$ 
\delta_0 \int  \la \xi \ra  \underline M(\xi) ^{-1} |     f (\xi )| ^2  d \xi 
\le  - \re \int \underline M(\xi)^{-1}  \underline L f (\xi)   f (\xi)  d\xi . 
$$
Hence: 
$$ 
\begin{aligned}
  \int  \la \xi \ra & \big( \delta_1 \underline  M(\xi) ^{-2s} + 
\lambda \delta_0 \underline M(\xi)^{-1} \big) |     f (\xi )| ^2  d \xi 
\\
& \le  - \re     \big(  \underline L f, f \big)_{\widetilde \HH^s}  + 
C \int 
  \underline M(\xi) ^{-2s} |     f (\xi )| ^2  d \xi . 
\end{aligned}
$$ 
We  choose $\lambda $ such that for all $\xi$
$$
C M(\xi) ^{-2s} \le \mez  \la \xi \ra   \big( \delta_1 \underline  M(\xi) ^{-2s} + 
\lambda \delta_0 \underline M(\xi)^{-1} \big)
$$
implying the inequality \eqref{coerc18}. 
 \end{proof}

\begin{rem}\label{fastdecay}
\textup{
Included in the bound \eqref{coercivs} is the observation
that both the first-order Chapman--Enskog approximation $\bar U_{NS}$
and the entire hierarchy of higher-order Chapman--Enskog correctors
lie in $\HH^s$, any $0<s<1$, something that is not immediately obvious. 
Indeed, looking closely at the inversion of $L_a$, we see that
they in fact decay at successively higher polynomial multipliers of the 
full Maxwellian rate.
}
\end{rem}
 

  \subsection{Comparison  }

    We consider the linearized operator $ L_a  g  = Q (a , g) + Q (g, a)$   at   
    $a  $, not necessarily nonnegative, 
 close to $\underline  M$. 
  
  \begin{prop}
  \label{perturblem}  
  For $s \in [\mez, 1[$ and $a \in \la \xi \ra^{- \mez} \HH^s$, 
  $L_a$ is bounded from $\la \xi \ra^{- \mez} \HH^s$
  to $\la \xi \ra^{\mez} \HH^s$. Moreover, there are constants $\delta > 0$,  $C > 0 $, $\lambda\ge 0$ 
  and $\eps_0 > 0$ 
  such that  for all 
  $f \in \la \xi \ra^{- \mez} \HH^s$
      \begin{equation}
      \label{actionLa}
  \big\|\la \xi \ra^{- \mez}  L_a f  \big\|_{\widetilde \HH^s}   
  \le C   \big\| \la \xi  \ra^\mez  \PP_\VV f  \big\|_{\widetilde \HH^s}   
+ C   \eps \big\| \la \xi  \ra^\mez  \PP_\UU f  \big\|_{\widetilde \HH^s},  
  \end{equation} 
 \begin{equation}
      \label{action2La}
\big|\,   \big( L_a  f    , f  \big)_{\widetilde \HH^s }  \big|    
  \le C   \big\| \la \xi  \ra^\mez  \PP_\VV f  \big\|^2_{\widetilde \HH^s}   
+ C   \eps^2  \big\| \la \xi  \ra^\mez  \PP_\UU f  \big\|^2_{\widetilde \HH^s},  
  \end{equation} 
  and 
  \begin{equation} 
    \label{coerccomp}
   \delta  \big\| \la \xi \ra^{\mez} \PP_\VV f  \big\|_{\widetilde \HH^s}^2  \le 
   -  \re \big( L_a  f    , f  \big)_{\widetilde \HH^s }   +      C  \eps^2  \big\| \la \xi \ra ^\mez \PP_\UU f  \big\|_{\widetilde \HH^s}^2 , 
  \end{equation} 
        with 
  \begin{equation} 
  \label{small1}
  \eps = \eps(a) :=   \big\| \la \xi \ra^\mez   (a  - \underline M) \big\|_{\widetilde \HH^s} . 
    \end{equation}

 \end{prop}
  
  In \eqref{coerccomp} the $\widetilde \HH^s$ scalar product has to be understood as the 
  integral 
  \begin{equation}
  \label{extscpr}
  \big( L_a  f    , f  \big)_{\widetilde \HH^s }  = 
   \int \big( \underline  M(\xi) ^{-2s} + 
\lambda \underline M(\xi)^{-1} \big)   L_a  f (\xi)   f (\xi)  d\xi ,  
  \end{equation}
  which is well defined since   $f \in \la \xi \ra^{- \mez} \HH^s$ and
   $L_a f \in \la \xi \ra^{\mez} \HH^s$.

  \begin{proof} That $L_a$ is bounded from  $ \la \xi  \ra^{- \mez} \HH^s$ to  $  \la \xi \ra^{ \mez} \HH^s$ for 
  all $s \in [\mez, 1[$,  follows 
 directly  from Section~\ref{collop}. 
  Moreover, 
  $$
   L_a f - \underline L f  = 
  Q( a - \underline M,  f )   
    +   Q(f,  a - \underline M )    
   $$
   and 
   $$
  \big\|\la \xi \ra^{- \mez} ( L_a f - \underline L f)  \big\|_{\widetilde \HH^s}   
  \le C \eps  \big\| \la \xi  \ra^\mez f \big\|_{\widetilde \HH^s}  . 
$$
Since $\underline L f = \underline L \PP_\VV f$, this implies \eqref{actionLa}.  
Since $L_a f $ and $\underline L f$ belong to $\VV$ and thanks to the definition 
of the modified scalar product  $\widetilde \HH^s$, 
$$
 \big( L_a  f - \underline L f    , f  \big)_{\widetilde \HH^s }= 
  \big( L_a  f  - \underline L f   , \PP_\VV f  \big)_{\widetilde \HH^s } = 
  O \Big( \eps   \big\| \la \xi  \ra^\mez f  \big\|_{\HH^s}   \big\| \la \xi  \ra^\mez \PP_\VV f \big\|_{\HH^s}\Big) .  
$$
 With \eqref{actionLa}, this implies \eqref{action2La} with a new constant $C$. 
With  Proposition~\ref{propcoercs}, this implies \eqref{coerccomp}. 
         \end{proof}

   \begin{rem}
   \textup{Since $\UU$ is finite dimensional, one can use any norm for 
   $\PP_\UU f$ in the estimates \eqref{actionLa} \eqref{action2La} and \eqref{coerccomp} above.  }
   \end{rem}

   \begin{rem} 
   \textup{We have in mind that $\eps (a)$ can be taken arbitrarily small.  
   This holds if $a = M_{\rmu}$ and 
   $\rmu$ is close to $\underline \rmu$ since , when $s < 1$,  } 
   $$
   \big\| \la \xi \ra^\mez   (M_\rmu - \underline M) \big\|_{\HH^s}^2 = 
   \int \la \xi \ra \Big|  M_\rmu -   \underline M \Big|^2 \underline M^{-2s} d \xi  \ \to \ 0 
   $$
  \textup{as $| \rmu - \underline \rmu  | \to 0$  by Lebesgue's dominated convergence theorem. } 
  \end{rem}


\section{Abstract formulation}\label{abframe}

We now rephrase the problem in a general framework,
for the square-root Maxwellian norm $\HH=\HH^{\frac{1}{2}}$ 
in which we carry out the main analysis.
We treat general weights in Section \ref{generalwt}, by
a bootstrap argument. 
Taking the shock speed equal to $0$  by frame-indifference,
we consider \eqref{Bode} as the abstract
standing-wave ODE  
\begin{equation}\label{relax1}
AU'= Q(U, U). 
\end{equation}
with 
\begin{equation}
A f(\xi)   = \xi_1 f(\xi)
\end{equation}
independent of $U$ (semilinearity of the Boltzmann equation), 
and $Q$ as in \eqref{colop}, \eqref{hs}.

\subsection{Bounds on the transport operator}\label{Abound}

The collision operator has been studied acting in spaces 
$\HH^s$ associated to our reference Maxwellian 
$\underline M$. We have the following evident facts regarding the transport operator $A$.

\begin{prop}
\label{symmA}
For $s \in [\mez, 1[$,  the operator $A$ is bounded from 
$\la \xi \ra^{- \mez} \HH^s$ to $\la \xi \ra^{- \mez}  \HH^{s}$  and (formally) self adjoint in $\HH^s$
as well as in $\widetilde \HH^s$. 
\end{prop}

%

 \subsection{Kawashima multiplier} 

We next construct a Kawashima compensator as in \cite{K,MeZ1,MTZ},
but taking special care that the operator remain bounded
in this infinite-dimensional setting.
   
   \begin{prop}
  There are $C$, $\delta > 0$,  $\lambda\ge 0 $ and
there is a finite rank  operator    $ K  \in \scL (\HH^{-1}, \HH^1)  $ 
  such that  such    $K $ is skew symmetric in $\widetilde \HH^s$     and 
 satisfies 
\begin{equation}
\label{pr5_1}
\Re  (  K  A   -  \underline L  )   \ge \delta  \la \xi \ra  \Id. 
\end{equation}
meaning that 
  \begin{equation}
  \label{lowbd}
  c  \big\| \la \xi  \ra^\mez f  \big\|_{\widetilde \HH^s}^2 \le 
  \re \big( (K A - L_a ) f    , f  \big)_{\widetilde \HH^s }  \le  C \big\| \la \xi  \ra^\mez f \big\|_{\widetilde \HH^s}^2 . 
  \end{equation} 
\end{prop}

\begin{proof}

{\bf a ) } We first check that the genuine coupling condition is satisfied, 
i.e. that there is no eigenvector of  $  A  $ in   
$\ker \underline L = \UU $. Indeed, using the basis $\phi_j$ of $\UU$ given in \eqref{orthobasis}, 
  an eigenvector of $A$ with eigenvalue $\tau $ in   $\UU  $ is a linear combination 
$\sum \alpha_j \phi_j$  such that the polynomial 
$$
(\xi_1 - \tau) \sum_{j=0}^4 \alpha_j \chi_j (\xi)   
$$
is identically zero. Equating to zero the term of degree 3 implies that $\alpha_4 = 0$.  
Equating to zero the coefficient of the terms of  degree 2  implies that $\alpha_j = 0$ 
for $j=1, 2, 3$,   and finally $\alpha_0= 0$. Thus  the property is  satisfied.

\medbreak

{\bf b) } We look for $K$ as 
 \begin{equation}
 K  =  \theta \big(  K_{11}      +   K_{12}  +   K_{21} \big) 
 \end{equation}
 with $\theta > 0$ a parameter to be chosen and 
 $$
  K_{12}    =  A_{12}  :=  \PP_\UU   A   \PP_\VV=A_{21}^*, 
  \qquad K_{21} = - K^*_{12} = - \PP_\VV  A \PP_\UU := - A_{21}  , 
 $$
 and 
 $$
 K_{11} = \PP_\UU  K_{11} \PP_\UU = - K_{11}^* . 
 $$
 Here $*$ means the adjoint with respect to the scalar product in $\HH$. 
 We have used Proposition~\ref{symmA}. 
 
Thus, with $A_{11} := \PP_\UU A  \PP_\UU$, 
 \begin{equation}
 \Re \PP_\UU K  A \PP_\UU =   \mez  [ K_{11} , A_{11} ]  +  A^*_{21} A_{21},  
 \end{equation} 
 The condition a)  means that 
 $A_{11}$ (restricted to $\UU $) has no eigenvector in $\ker A_{21}=
\ker A^*_{21} A_{21}$, with  $A^*_{21} A_{21}$ symmetric positive 
semidefinite and $A_{11}$ symmetric.
Since $\dim \UU $ is finite (equal to $5$),  
 this implies by the standard, finite-dimensional construction of
Kawashima et al \cite{K} that one can choose $K_{11}$  such that 
 $\Re \PP_\UU  K A  \PP_\UU $ is definite positive on $\UU$: there is $c_1 > 0$ such that
 \begin{equation}
 \big( \Re \PP_\UU K A \PP_\UU f  ,   \PP_\UU  f  \big)_\HH \ge c_1 \big\| \PP_\UU f \big\|^2_\HH. 
 \end{equation}
 Moreover, since $\dim \UU $ is finite, there is another $c_1 > 0$ such that 
  \begin{equation}
 \big( \Re \PP_\UU K A \PP_\UU f  ,   \PP_\UU  f  \big)_\HH \ge c_1 \big\| \PP_\UU f \big\|^2_{\HH^1}. 
 \end{equation}
 Thus, using Proposition~\ref{coercivmez}  for $a = \underline M$: 
 $$
\begin{aligned}
 \Re \big(( KA     - \underline L  ) f, f \big)_\HH  \ge \theta c_1 & \big\| \PP_\UU f \big\|^2_{\HH^1}  + 
 c  \big\| \PP_\VV f  \big\|^2_{\HH^1} 
 \\ &  -  \theta   C   \big\|   f  \big\| _{\HH^1}   \big\| \PP_\VV f \big\|_{\HH^1} 
 \end{aligned}
 $$
 with 
 $$
 C = \big\|  K_{11}  \PP_\UU  A  \PP_\VV  \big\|  +     \big\|  K_{12}  \PP_\VV A \PP_\VV  \big\|
 +  \big\|  K_{21}  \PP_\UU  A  \PP_\UU  \big\| +  \big\|  K_{21}  \PP_\UU  A  \PP_\VV  \big\|
 $$
 where the norms are taken in $\scL (\HH^1; \HH^{-1} )$.  
 All these operators have finite rank $\le n$ and  are bounded. 
 Thus if $\theta  $ is small enough,  this shows that
 $\Re (K  A  - \underline L)$ is definite positive  in the sense of \eqref{lowbd}.  
 Using the perturbation Lemma~\ref{perturblem} implies that the estimate remains
 true for $a$ satisfying \eqref{small1}. 
\end{proof} 

\begin{rem}\label{Kcomp}
\textup{
The construction above, by reduction to the equlibrium manifold, 
is essentially different from the original proof of \cite{K}
in the finite-dimensional case, which would yield a symmetrizer
of infinite rank.  
The advantage of finite rank is that we need not worry about boundedness
of the operator.
We note that this is related to methods in the Boltzmann
literature in which the Kawashima compensator is replaced by
estimates on a reduced Chapman-Enskog approximation such as
the Grad 
$13$ moments model or the Navier--Stokes approximation,
again to avoid possible boundedness issues; 
see, e.g., \cite{G,LY}.
}

\textup{ See also the related construction
of \cite{GMWZ} in the case that $u$ is scalar, for which $K_{11}$
may be taken equal to zero.
We note that we could apply the same reduction argument to the
reduced problem and proceed by iteration to this scalar case, thus
obtaining an alternative proof in the finite-dimensional case as well.
}
\end{rem}

 
 \subsection{Reduction to bounded operators}
In the hard-sphere case \eqref{hs}, we may
rescale the equations to obtain a problem involving only 
{bounded operators}.
%
 We have $\HH^1 \subset \HH \subset \HH^{-1}$, bounded operators from 
 $\HH^1$ to $\HH^{-1}$ and we work with the scalar product of $\HH$. 
 We can multiply the equations on the left by $\la \xi \ra^{-1}$: 
 if $A \in \scL (\HH^1; \HH^{-1})  $  then 
 $$
\hat  A := \la \xi \ra^{-1} A \in \scL (\HH^1; \HH^1)
$$
 and 
 $$
 \big( \hat  A f , f \big)_{\HH^1}  =   \big( A f, f\big)_\HH  
 $$
 so that if $A$ is symmetric  in $\HH$, $\hat  A$ is symmetric in $\HH^1$. 
 
 Equivalently, we can make the change of variable 
 $f \mapsto \tilde f =  \la \xi \ra^{\mez} f $  from 
 $\HH^1$ to $\HH$ and define
 $$
 \tilde A \tilde f  :=  \la \xi \ra^{- \mez}  A f  =   \la \xi \ra^{- \mez}  A   \la \xi \ra^{- \mez} f.
 $$
 Then $\tilde A \in \scL (\HH; \HH)$ and $\tilde A$ is symmetric in $\HH$ if 
 $A$ is. 

By Corollary \ref{Qbd}, the corresponding collision operators
 $$
\hat  Q(f,f) := \la \xi \ra^{-1} Q(f,f) 
$$
and
 $$
 \tilde Q( \tilde f, \tilde f)  := \la \xi \ra^{- \mez}  Q(   \la \xi \ra^{- \mez} f,
\la \xi \ra^{- \mez} f)
 $$
by Corollary \ref{Qbd} are bounded as well:
$ \hat Q\in \scB (\HH^1; \HH^1) \; {\rm  and }\; \tilde Q\in \scB (\HH; \HH), $
where $ \scB (H; H' )$ denotes the space of continuous
bilinear forms from $H\to H'$, i.e., $B\in \scB(H;H')$ if and only if
\be\label{scB}
 \big\| B (g, h) \big\|_{H'}  \le  C_s 
  \big\| g \big\|_{H}  \big\| h \big\|_{H}  .
\ee


 \subsection{The framework}

At this point, we have reduced to the following abstract problem,
with semilinear structure quite similar to that treated in the
finite-dimensional analysis of \cite{MeZ1}.
Working in $\HH$ with operators $\tilde A$ and $\tilde Q$ and dropping
tildes,
we study the standing-wave ODE
\begin{equation}\label{relax}
AU'= Q(U, U),
\end{equation}
with $U$ taking its values in an infinite dimensional space $\HH$. 

\subsubsection{Assumptions on the full system}
We make the following assumptions, verified above 
for the Boltzmann equation in the hard-sphere case
with $A$, $Q$ replaced by $\tilde A$, $\tilde Q$.

\begin{ass}\label{ass1}
(i)  $A$ is a  bounded self adjoint operator in a (real) Hilbert space $\HH$; 

(ii)  There is an orthogonal splitting 
$\HH = \UU \oplus \VV$  with  $\UU$ finite dimensional

(iii)  $Q$ is bilinear, continuous (in sense \eqref{scB}) 
and symmetric (in $U$) from $\HH \times \HH$ to $\VV$.  
\end{ass}

 For $U \in \HH$, we denote by  
 $L_U      $   the bounded operator  $ V \mapsto  2  Q(U, V)$, that is the differential 
 of $Q(U, U)$. 
We denote by $\PP_\UU$ and $\PP_{\VV}$ the orthogonal 
projectors on $\UU$ and $\VV$ respectively. 
We use the notations $U = u + v $, with $u = \PP_{\UU} U $
and $v = \PP_{\VV} U $.

  \begin{ass}\label{refass}
  We are given a reference state $ \underline U$ (in a smaller space $\MM  \subset \HH$) 
  such that 
   $\underline L = L_{\underline U}$ is self adjoint  with kernel $\UU$  and 
  $\underline L$ is definite negative on $\VV$. 
  
  \end{ass}

\begin{lem}\label{coerc1}
 There  are  $\delta > 0$, $\eps_0$   and $C \ge $ such that  for $a \in \MM$
  and $U \in \HH$: 
  \begin{equation}
  \label{infLa}
-    \re \big( L_a U, U \big)_{\HH}  \ge \delta \big\| \PP_{\VV} U \big\|^2_{\HH} 
-  C \eps  \big\| \PP_{\UU} U \big\|_{\HH}  \big\| \PP_{\VV} U \big\|_{\HH} 
  \end{equation} 
  provided that 
  \begin{equation}
  \label{small2}
  \big\| a - \underline M \big\|_{\HH}  \le \eps  \le \eps_0. 
  \end{equation}
  
  \end{lem} 

\begin{proof} By continuity of $Q$, \eqref{scB}, there is $C$ such that 
 \begin{equation}
  \label{compar}
   \big\|L_a  U  - \underline L u  \big\|_{\HH}  \le C 
    \big\| a - \underline M  \big\|_{\HH} \big\| U \big\|_{\HH} . 
  \end{equation} 
  Moreover, there is $\delta > 0$ such that 
  $$
  -    \big( \underline L U, U \big)_{\HH}  = 
   -    \big( \underline L U, \PP_\VV U \big)_{\HH}
   =  -    \big( \underline L \PP_\VV U, \PP_\VV U \big)_{\HH}  \ge \delta \big\| \PP_{\VV} U \big\|^2_{\HH} . 
  $$
  Since 
  $$
    \big( L_a U, U \big)_{\HH} =   \re \big( L_a U, \PP_\VV U \big)_{\HH} 
    $$
    the lemma follows. 
  \end{proof}


\begin{lem}\label{smoothman}
In an $\HH$-neighborhood of $\underline U$,
the zero set of $Q$ is given by a smooth (indeed $C^\infty$) 
manifold $M=\{U: \; v=v_*(u)\}$ with $v_*: \UU\to \VV$ smooth.
\end{lem}

\begin{proof}
Assumption \ref{refass}
and the Implicit Function Theorem, together with the observation that
$Q$ as a continuous biinear form (in sense \eqref{scB}) is $C^\infty$
in the Frechet sense.
\end{proof}
 
We further assume the Kawashima condition established in Proposition 
\ref{Kcomp}.

\begin{ass}
\label{assinf3}
There is a  skew symmetric bounded operator  
$  K   \in \scL (\HH)$  and  
a constant $\gamma > 0$ such that
\begin{equation}
\label{pr5}
\Re   K A -   \underline L    \ge \gamma \Id. 
\end{equation}
\end{ass}

Using \eqref{compar}, this implies
 
 \begin{lem}
 \label{kawa}
 There are $\gamma > 0$ and $\eps_ 0> 0$ such that for 
 $a \in \HH$ satisfying \eqref{small2} and  
 $U \in \HH$: 
   \begin{equation}
  \label{infkawa}
   \re \big( ( K A   -  L_a) U, U \big)_{\HH}  \ge \gamma \big\|   U \big\|^2_{\HH} 
  \end{equation} . 
 
 \end{lem} 

\subsubsection{Assumptions on the reduced system}\label{ceframe}

Coordinatizing $U \in \HH$ as 
\begin{equation}
\label{pr7_1}
 U=\begin{pmatrix} u\\v \end{pmatrix}
= \begin{pmatrix} P_\UU U\\P_\VV U \end{pmatrix},
 \end{equation}
we have 
\begin{equation}\label{pr9}
 A= \begin{pmatrix} A_{11}& A_{12}\\A_{21}& A_{22}\end{pmatrix},
\quad Q=\begin{pmatrix} 0\\q(U, U)\end{pmatrix}
\end{equation}
 with $A_{11} \in \scL (\UU; \UU) $,  $A_{12} \in \scL (\VV; \UU) $ etc. 
 We use the notation 
\begin{equation}
\label{pr10}
h(u,v):= A_{11}u + A_{12}v \in \UU.
\end{equation}
 Finally, the equilibria are parametrized by $u$: 
  \begin{equation}
\label{pr7}
 M(u)  =\begin{pmatrix} u\\ v_*(u)  \end{pmatrix},  
 \end{equation}
where $v_*$ is the smooth mapping from a neighborhood of 
$\underline u$ to a neighborhood of $\underline v = v_* (\underline u)$
in $\VV$, as described in Lemma \ref{smoothman}.
 
 Recall from \cite{Y} that the reduced, Navier--Stokes type equations
obtained by Chapman--Enskog expansions are
\begin{equation}\label{ce} 
h_*(u)'= (b_*(u)u')',
\end{equation}
where
\begin{eqnarray}
\label{cevalues} 
&h_*(u)&:=h(u,v_*(u))=A_{11}u+A_{12}v_*(u),
\\
\label{bstar}
&b_*(u)&:= -A_{12} c_* (u) 
\end{eqnarray}
with
\begin{equation}
\label{cstar}
\begin{aligned}
c_* (u) := \partial_v  &q^{-1} (u,v_*(u))&\\
&\Big(A_{21} +  A_{22}d v_*(u)    - d v_*(u) (A_{11}+ A_{12} d v_*(u) 
\Big). 
\end{aligned}
\end{equation}
 Note also, by the Implicit Function Theorem, that 
$ d v_*(u) = -   \partial_v q^{-1}\partial_u q (u, v_*(u)). $

An important property of the Chapman-Enskog approximation, following
either by direct computation or by coordinate-independence of the
physical derivation,
is that the form \eqref{cevalues}--\eqref{cstar}
of the equations is coordinate invariant, changing tensorially
with respect to constant linear coordinate changes; 
moreover, the change in functions $h_*$, $b_*$ due to a constant linear
coordinate change may be computed directly from \eqref{ce}
using the coordinate change in $u$ alone.
From this we find in the Boltzmann case that \eqref{ce} is equivalent
through a constant linear coordinate change to the Navier--Stokes
equations \eqref{NS1} with monatomic ideal gas equation of state and
viscosity and heat conduction coefficients satisfying \eqref{coeffs}.
We make the following assumptions on the reduced system, verified
for the Navier--Stokes equations (hence satisfied for the Boltzmann
equation) in \cite{MaZ3}.

\begin{ass}\label{goodred}
	(i) There exists $s(u)$ symmetric positive
definite such that $s\, dh_*$ is symmetric and $sb_*$ is
symmetric positive semidefinite.

(ii) There is no eigenvector of $dh_*$ in $\ker b_*$.

(iii) The matrix $b_*(u)$ has constant left kernel.

(iv) For all values of $u$,  $\ker \pi_* dh_*(u) \cap \ker b_*(u) = \{ 0 \}$,
where $\pi_*(u)$ is the zero eigenprojection associated with $b_*(u)$.   

\end{ass}

Finally, we assume that the classical theory of weak shocks 
can be applied to  \eqref{ce}, requiring that the 
flux $f_*$ have a  genuinely nonlinear eigenvalue near $0$.

\begin{ass}\label{profass}
 
In a neighborhood $\cU_*$ of   a given base state $u_0$,  
$dh_*$ has a simple eigenvalue $\alpha$ near zero, with $\alpha (u_0) = 0$, and such that the
associated hyperbolic characteristic field is genuinely 
nonlinear, i.e., after a choice of orientation, $\nabla \alpha \cdot r(u_0) <  0$, where
$r$ denotes the eigendirection associated with $\alpha$.
 
\end{ass}

\begin{rem}
\textup{
As discussed in \cite{Y}, Assumptions \ref{goodred}(i)--(ii)
hold in great generality.
Assumptions \ref{goodred}(iii)-(iv) must be checked in individual cases.
}
\end{rem}

 \subsection{The basic estimate}\label{basic}

With these preparations, we can establish existence by an argument
almost identical to that used in \cite{MeZ1} to treat the finite-dimensional
case: indeed, somewhat simpler.
The single difference is that in carrying out the basic symmetric energy
estimates controlling microscopic variables
 we do not attempt to exactly symmetrize
$L_a$ at each $x$ value as was done in \cite{MeZ1}, but only use
the fact that each $L_a$ is approximately symmetric by construction.
This is important in the infinite-dimensional case, since exact
symmetrization can (and does in the Boltzmann case) introduce unbounded
commutator terms that wreck the argument.
%
To isolate this important technical point, we
carry out the key estimate here, before describing the rest of the
argument.
 
 We consider the equation 
 \begin{equation}
 \label{lineareq}
 A \D_x U  -  L_a U  = F   
 \end{equation}
with $a = a(x)$  satisfying
\begin{eqnarray}
\label{small3}
   \big\| \D_x^k \big( a (x) - \underline M\big)  \big\|_{\HH}  \le 
   C_k \eps^{k+1}  e^{ - \eps \theta \la x \ra }. 
\end{eqnarray} 

We assume that 
\begin{equation}
\label{specialF}
\PP_\UU F  = \eps h  + \D_x f. 
\end{equation}

\begin{lem}
\label{basicest}
There is a constant $C$ such that for $\eps$ sufficiently small, 
one has 
\begin{equation}\label{sharph1eq}
\| U'  \|_{L^2} + \| \PP_{\VV} U \|_{L^2}  \le C \big(     
\|  f  \|_{H^2}   +   \| h \|_{H^1} + \| g  \|_{H^1}   
+ \eps \| \PP_{\UU} U \|_{L^2} \big). 
\end{equation}
\end{lem} 

 Here, the norms $L^2$, $H^1$ etc denote the norms in 
 $L^2 (\RR ; \HH)$, $H^1 (\RR ; \HH)$ etc.

\begin{proof}
 
 Introduce the symmetrizer
\begin{equation} 
\label{def615}
\cS = \D_x^2    + \D_x \circ  K   -  \lambda  \Id .  
\end{equation}
One has 
$$
\begin{aligned}
& \re \D_x^2  (A\D_x  -      L_a)  =       -  \re   \D_x  \circ  L_a  \circ    \D_x  -  \re \D_x  \circ L_{\D_x a}    
\\
& \Re \D_x \circ K  (A\D_x -  L_a )  =  \D_x \circ  \Re KA \circ  \D_x  -   \re \D_x \circ K  L_a  
\\
&  \Re     ( A\D_x -    L_a  )    =     -  \Re  L_a
\end{aligned}
$$
where 
$\re T = \mez (T + T^*)$ and the adjoint is taken 
in $L^2 (\RR ; \HH)$. We   have used that $   [\D_x, L_a] = L_{\D_x a} $  by linearity of  $L$ with respect 
to $a$. 
Thus 
$$
\begin{aligned}
\Re \cS  \circ ( A\D_x -    L_a )  =  & \D_x \circ (\Re AK  - L_a ) \circ \D_x   -
   \lambda \re L_a     
\\ 
&    -   \Re \D_x \circ  L_{\D_x a}  -   \Re \D_x \circ K  L_a. 
\end{aligned}
$$
Therefore, for  $U \in H^2 (\RR)$,  \eqref{infLa}, \eqref{infkawa}    and the continuity 
of $K$ and $Q$ imply that 
$$
\begin{aligned}
   \Re ( \cS   F , U)_{L^2}   \ge & \ \gamma  \| \D_x U \|^2_{L^2} + 
\lambda  \big( \delta   \|  \PP_\VV U  \|^2_{L^2}   -  C \eps  \|  \PP_\UU U  \|_{L^2} 
 \|  \PP_VV U  \|_{L^2}\big) 
\\
& -  C  \| \D_x a  \|_{L^\infty} \|  U \|_{L^2} \| \D_x U \|_{L^2}  -  C  \|\D_x  U \|_{L^2} 
\|  L_a U  \|_{L^2} . 
\end{aligned}
$$
We note that 
\begin{equation}
\label{almostblock}
L_a U =  \underline L \PP_{\VV} U  +  ( L_a - \underline L) U 
\end{equation}
Therefore, 
\begin{equation}
\label{almostblock2}
\big\|  L_a U \big\|_{HH}    \lesssim  \big\| \PP_\VV U \big\|_{HH}   + 
\eps \big\|  \PP_\UU U \big\|_{HH} 
\end{equation}
Taking  $\lambda$ large enough
and using \eqref{small3} 
yields
$$
\| U' \|^2_{L^2} + \| \PP_\VV U  \|_{L^2}^2 \lesssim     \Re ( \cS   F , U)_{L^2}   +  
\eps  \| \PP_\UU U \|_{L^2} \big(  \| \PP_\VV U \|_{L^2} +  \| U' \|_{L^2} \big) .   
$$

In the opposite direction,   
$$
\begin{aligned}
   \Re ( \cS F , U)_{L^2}   \le      & \| \D_x  U \|_{L^2} \big( 
 \| \D_x F  \|_{L^2}  + \| K \|  \| F \|_{L^2} \big) 
\\
& + \lambda \big(\eps    \|  h  \|_{L^2}   \| \PP_\UU  u \|_{L^2}  
  +      \|  f \|_{L^2}   \| \PP_UU  \D_x U   \|_{L^2}  
\\
& \hskip 4cm +       \| \PP_\VV F   \|_{L^2} \|  \PP_\VV U\|_{L^2}  \big).  
\end{aligned}
$$
The estimate \eqref{sharph1eq}  follows provided that $\eps$  is small enough. 

This proves the lemma under the additional assumption that $U \in H^2$. 
When $U\in H^1$, the estimate follows using Friedrichs mollifiers. 
\end{proof}

%
 
\section{Basic $L^2$ result}\label{basicresult}

We now describe a simpler version of our main result, carried
out in the $L^2$ norm $\HH$.
For clarity of exposition, we carry out the entire argument in
this more transparent context, indicating afterward in Section
\ref{othernorms} how to extend to the general (pointwise,
higher weight) norms described in Theorem \ref{mainthm}.

\subsection{Chapman--Enskog approximation} \label{CEapprox}

Integrating the first equation of \eqref{relax} and 
noticing that the end states $(u_\pm, v_\pm) $ must be equilibria and thus satisfy 
$v_\pm = v_*(u_\pm)$,  we obtain
\begin{equation}\label{intprof}
\begin{aligned}
A_{11} u + A_{12} v &=  f_*(u_\pm),\\
A_{21}u'+ A_{22}v'&=q(u,v).
\end{aligned}
\end{equation}

Because $f$ is linear,  the first equation reads
\begin{equation}\label{T1}
f_*(u) + A_{12}(v-v_*(u))  = f_*(u_\pm).
\end{equation}
The idea of Chapman--Enskog approximation is that 
$v -  v_*(u)$ is small  (compared to the fluctuations $u - u_\pm$). Taylor expanding the second equation, we obtain
$$
\begin{aligned}
(A_{21}+A_{22}dv_*(u)) u' + A_{22}(v-v_*(u))'&=
\partial_v q(u,v_*(u))(v-v_*(u)) 
\\
&\quad + O(|v-v_*(u)|^2),    
\end{aligned}
$$
or inverting $\partial_v q$  
\begin{equation}\label{T2}
\begin{aligned}
v-v_*(u)  =  & \ \partial_v q^{-1} (u, v_*(u))   \big(A_{21} + A_{22} d v_*(u) \big) 
 u' \\
& + O(|v-v_*(u)|^2) + O(|(v-v_*(u))'|). 
\end{aligned}
\end{equation}
The derivative of \eqref{T1} implies that 
$$
 \big( A_{11} u + A_{12} dv_* (u) \big) u'  =  O(|(v-v_*(u))'|) . 
$$
 Therefore,  \eqref{T2} can be replaced by 
 \begin{equation}\label{T2b}
\begin{aligned}
v-v_*(u)&= c_* (u) u' 
   + O(|v-v_*(u)|^2) + O(|(v-v_*(u))'|),
\end{aligned}
\end{equation}
where  $c_*$ is defined at \eqref{cstar}. 
Substituting in \eqref{T1},  we thus obtain the approximate viscous profile ODE
\begin{equation}\label{approxprof}
\begin{aligned}
b_*(u)u'&= f_*(u) -f_*(u_\pm) 
+ O(|v-v_*(u)|^2) + O(|(v-v_*(u))'|),
\end{aligned}
\end{equation}
where $b_*$ is as defined in \eqref{bstar}.

Motivated by \eqref{T2}--\eqref{approxprof}, we define an approximate
solution $(\bar u_{NS}, \bar v_{NS})$ of \eqref{intprof} by choosing 
$\bar u_{NS}$  as a solution of 
\begin{equation}
\label{NS}
b_*(\bar u_{NS})\bar u_{NS}' = f_*(\bar u_{NS}) -f_*(u_\pm),
\end{equation}
and $\bar v_{NS}$  as the first approximation given by \eqref{T2} 
\begin{equation}
\label{NSv}
\begin{aligned}\bar v_{NS} -v_*(\bar u_{NS})  =    c_* (\bar u_{NS}) 
 \bar u_{NS}'.
\end{aligned}
\end{equation}

Small amplitude shock profiles   solutions of \eqref{NS}  are constructed 
using the center manifold  analysis of \cite{Pe}
under conditions (i)-(iv) of Assumption \ref{goodred}; see discussion  in 
\cite{MaZ5}.

\begin{prop}[\cite{MaZ5}]\label{NSprofbds} Under Assumptions~\ref{goodred}
and \ref{profass}, 
in a neighborhood of 
$(u_0, u_0)$ in $\RR^n \times \RR^n$, 
there is a smooth  manifold $\cS$ of dimension $n$  passing through $(u_0, u_0)$,  such that 
for $(u_-, u_+) \in \cS$ with   amplitude $\eps:=|u_+ -u_-| > 0$ 
sufficiently small, and direction $(u_+-u_-)/\eps $ sufficiently close
to $r(u_0)$,   the zero speed shock profile equation   \eqref{NS} has  a unique (up to translation) 
solution   $\bar u_{NS}$ in $\cU_*$. 
The shock profile is necessarily of {\rm Lax type}: i.e., with
dimensions of the unstable subspace of $dh_*(u_-)$
and the stable subspace of $dh_*(u_+)$ summing to one plus the
dimension of $u$, that is $n+1$.

Moreover, 
there is  $\theta>0$ and for all $k$ there is $C_k $ independent of $(u_-, u_+) $ and $\eps$,   
such that 
\begin{equation}\label{NSbds}
|\partial_x^k (\bar u_{NS}-u_\pm)|\le C_k \eps^{k+1}e^{-\theta \eps|x|},
\quad x\gtrless 0. 
\end{equation}
\end{prop}

We denote by 
 $\cS_+$  the set  of $(u_-, u_+) \in \cS $  with  amplitude $\eps:=|u_+ -u_-| > 0$ 
sufficiently small  and direction $(u_+-u_-)/\eps $ sufficiently close
to $r(u_0)$ such that the profile $\bar u_{NS}$ exists.  
Given $(u_-, u_+) \in \cS_+  $ with associated profile $\bar u_{NS}$, 
we define $\bar v_{NS} $ by \eqref{NSv} and 
    \begin{equation}
    \label{NSU}
    \bar U_{NS} := (\bar u_{NS}, \bar v_{NS}).
    \end{equation}
    It  is an approximate solution of \eqref{intprof} in the following sense: 

\begin{cor}\label{redbds}
For   $(u_-, u_+) \in \cS_+$, 
\begin{equation} 
\label{exactfeq}
 A_{11} \bar u_{NS}  + A_{12} \bar v_{NS} - f_*(u\pm) = 0
 \end{equation} 
 and 
 $$
 \cR_v  := A_{21}\bar u_{NS}' + A_{22}\bar v_{NS}'-q(\bar u_{NS},\bar v_{NS})  
$$ 
satisfies 
\begin{equation}\label{eq:resbds}
|   \D_x^k \cR_v  (x) |  \le  C_k \eps^{k+3}e^{-\theta \eps|x|} , \quad x\gtrless 0
\end{equation}
where $C_k$   is  independent of $(u_-, u_+) $ and  $\eps=|u_+ -u_-|$. 
 \end{cor}

\begin{proof}
Given the  choice of $\bar v_{NS}$, the first equation is a rewriting of the profile equation  
\eqref{NS}.  

Next, note that 
$$
 \bar v_{NS}  - v_* (\bar u_{NS}) = O  (| \bar u'_{NS} | ) , 
 \quad  \big( \bar v_{NS}  - v_* (\bar u_{NS})\big)'   = O  (| \bar u''_{NS} | ) + O  (| \bar u'_{NS} |^2 ), 
$$
where here  $O ( \cdot )$ denote  smooth functions of $\bar u_{NS}$ and its derivatives, 
which vanish as indicated. With similar notations, 
the Taylor expansion of $q$ and the definition of $\bar v_{NS}$ show that 
$$
\begin{aligned}
\cR_v       = &  O(| \bar v_{NS} -v_*(\bar u_{NS})|^2) + O(|(\bar v_{NS} -v_*(\bar u_{NS}))'|)
\\
& +   d v_* (\bar u_{NS}   \big(A_{11} + A_{12}  dv_* (\bar u_{NS}) \big )  \bar  u'_{NS}  . 
\end{aligned}
$$
Moreover, 
$$
\begin{aligned}
\big(A_{11} + A_{12}  d v_* (\bar u_{NS}) \big)  \bar  u'_{NS}  &= \big(  f_* (\bar u_{NS}) \big)' 
=   \big( b_* (\bar u_{NS}) \bar u'_{NS}) \big)'   
   \\
   & = O(|\bar u_{NS}'|^2) +    O(|\bar u_{NS}''|). 
\end{aligned}
$$
 This implies that  
 $$
\cR_v =   O(|\bar u_{NS}'|^2) +    O(|\bar u_{NS}''|).  
$$
  satisfies the estimates stated in \eqref{eq:resbds}. 
\end{proof}

\begin{rem}\label{correctorrmk}
\textup{
One may check that 
if we did not include the correction from equilibrium on the righthand
side of \eqref{NSv}, taking instead the simpler prescription
$\bar v_{NS} =v_*(\bar u_{NS})$ as in \cite{LY}, then 
the residual error that would result in \eqref{exactfeq} would be
too large for our later iteration scheme to close.
This is a crucial difference between our analysis and the analysis
of \cite{LY}.
The prescription $\bar U_{NS} $ corresponds to the 
first-order Chapman--Enskog approximation in both variables,
$u$ and $v$ together.
}
\end{rem}

 %

\subsection{Basic $L^2$ result}\label{results}

We are now ready to state the basic $L^2$ version of our main result.
Define a base state $U_0=(u_0,v_*(u_0))$ and a
neighborhood $\cU=\cU_*\times \cV$.   

\begin{prop}\label{main}
Let Assumptions (SS), (GC), and  \ref{goodred} hold on the
neighborhood $\cU$ of $U_0$, with $Q\in C^{\infty}$. 
Then, there are $\eps_0 > 0$  and 
$\delta > 0$ such that for $(u_-, u_+) \in \cS+$ with  amplitude $\eps:=|u_+-u_-| \le \eps_0$,   the standing-wave equation 
\eqref{relax} has a solution   
$\bar U$ in $\cU$, 
 with associated Lax-type 
equilibrium shock $(u_-,u_+)$, satisfying for all $k  $: 
\begin{equation}\label{finalbdsold}
\begin{aligned}
\big|\partial_x^k (\bar U- \bar U_{NS})\big|
&\le C_k \eps^{k+2}e^{-\delta  \eps|x|},\\
|\partial_x^k (\bar u-u_\pm)|&\le C_k \eps^{k+1}e^{-\delta \eps|x|},
\quad x\gtrless 0,\\
\big|\partial_x^k (\bar v-v_*(\bar u)\big|
&\le C_k \eps^{k+2}e^{-\delta  \eps|x|},\\
\end{aligned}
\end{equation}  
where $\bar U_{NS}=(\bar u_{NS}, \bar v_{NS})$ is the 
approximating Chapman--Enskog profile defined in \eqref{NSU}, and
$C_k$ is independent of  $\eps$. 
Moreover, up to translation, this solution is unique
within a ball of radius $c\eps$ about $\bar U_{NS}$ in norm 
$\|\cdot\|_{L^2}+\eps^{-1}\|\D_x \cdot\|_{L^2}
+ \eps^{-2}\|\D_x^2 \cdot\|_{L^2} $, for $c>0$ sufficiently small.
(For comparison, $\bar U_{NS}-U_\pm$ is order $\eps^{1/2}$ in this norm,
by \eqref{finalbds}(ii)--(iii).)
\end{prop}

\section{Outline of the proof}\label{outline}

We describe now the main steps in the proof of Proposition \ref{main},
exactly following the finite-dimensional analysis of \cite{MeZ1}.
 
\subsection{Nonlinear perturbation equations}
Defining the perturbation variable $U:= \bar U- \bar U_{NS}$,
and expanding about $\bar U_{NS}$,
we obtain from \eqref{intprof} the nonlinear perturbation equations
\begin{eqnarray}
  A_{11} u + A_{12} v  & = &0 
\\
  A_{21} u' + A_{22} v' -    d q (\bar U_{NS}) U & = & - \cR_v + N( U) 
\end{eqnarray} 
where the remainder $N(U) $ is  a smooth function of $U_{NS} $ and $U$, 
vanishing at second order at  $U =0$: 
\begin{equation}\label{Nbds}
N(U)= \cN(\bar U_{NS}, U)  = O(|U|^2). 
\end{equation}
We push the reduction a little further, using that 
\begin{equation}
\label{defM}
M := 
dq(\bar u_{NS}, \bar v_{NS})-
dq(\bar u_{NS}, v_*(u_{NS})) =
O(|\bar v_{NS}-v_*(\bar u_{NS})|). 
\end{equation}
 Therefore the equation reads 
 \begin{equation}
 \label{intpert}
 \begin{aligned}
\cL_*^\eps U:=&
\begin{pmatrix}0 & 0 \\ A_{21} & A_{22} 
\end{pmatrix}U'
+
\begin{pmatrix}A_{11} & A_{12} \\
- Q_{21}  & - Q_{22} \end{pmatrix}U
\\ 
= &
\begin{pmatrix} 0  \\ - \cR_v + MU + N(U)\end{pmatrix}
\end{aligned}
\end{equation}
where 
\begin{equation}
Q_{21} = \D_u q (\bar u_{NS}, v_* (\bar u_{NS})), 
\quad 
Q_{22} = \D_v q (\bar u_{NS}, v_* (\bar u_{NS})). 
\end{equation}

Differentiating the first line, it implies that 
\begin{equation}\label{pert} 
L_*^\eps U:=
AU'-dQ(\bar u_{NS},v_*(\bar u_{NS}))U=
\begin{pmatrix}   0  \\ - \cR_v + MU + N(U)\end{pmatrix}. 
\end{equation}
  
The linearized operator $A\partial_x - dQ(\bar U)$
about an exact solution $\bar U$ of the profile equations
has kernel $\bar U'$, by translation invariance, so is not invertible.
Thus, the linear operators $L_*^\eps $ and $\csL$ are not 
expected to be invertible,
and we shall see later that they are not.  
Nonetheless, one can check that $\csL$ is surjective in Sobolev spaces and    define a right inverse
$\csL \csLd\equiv I$, or solution operator
$(\cL_*^\eps)^\dagger$ of the equation 
\begin{equation}
\label{neweq}
\cL_*^\eps U=F:=  \begin{pmatrix}f\\g\end{pmatrix}, 
\end{equation}  
as recorded by   Proposition~\ref{invprop} below. 
Note that $\sL$ is not surjective   because   the first equation requires 
a zero mass condition on the source term. This is why we solve 
the integrated equation \eqref{intpert} and not \eqref{pert}. 

To  define the partial inverse $\csLd$, we  specify  one solution of 
\eqref{neweq} by adding the co-dimension one  internal 
condition: 
\begin{equation}\label{phasecond}
\ell _\eps \cdot u(0)  =0 
\end{equation}
where $\ell_\eps$ is a certain unit vector to be specified below. 

\begin{rem}\label{ellchoice}
\textup{ There is a large flexibility in the choice of $\ell_\eps$.
Conditions like \eqref{phasecond}  are  known to fix the indeterminacy in the 
resolution of the linearized profile equation  from \eqref{NS}
and  it remains well adapted in the present context,  see section~\ref{linCEestimates} below. 
A possible choice,  would be to choose $\ell_\eps$ independent of 
$\eps$ and parallel to the left  eigenvector of $ dh_* (u_0)$ for the eigenvalue $0$
(see  Assumption~\ref{profass}),
which, together with the asymptotics of Proposition \ref{NSprofbds}, gives
\be\label{normalization}
\ell_\eps \cdot \bar U_{NS}'(0)\sim \eps^2 \ne 0.
\ee
    }
\end{rem}


\subsection{Fixed-point iteration scheme}

The coefficients  and the error term $\cR_v$ are smooth functions of 
$\bar u_{NS}$ and its derivative, thus behave like smooth functions of 
$ \eps x$. Thus, it is natural to solve the equations in spaces which reflect 
this scaling. We do not introduce explicitly the change of variables
$\tilde x = \eps x$, but introduce norms which correspond to the usual $H^s$ norms 
in the $\tilde x $ variable : 
\begin{equation}
\label{defnorm}
\|f \|_{H^s_\eps} =  
\eps^{\mez}  \|f \|_{L^2}+
\eps^{-\mez }\|\partial_x f\|_{L^2}+ \dots + 
\eps^{\mez-s}\|\partial_x^s f\|_{L^2}.
\end{equation}
We also introduce weighted spaces and norms, which encounter for the exponential 
decay of the source and solution: introduce the notations.  
\begin{equation}
\label{modx}
<x>:= (x^2+1)^{1/2}
\end{equation}
For  $\delta \ge 0$ (sufficiently small), we denote by $H^s_{\eps, \delta}$ the space of 
functions $f$ such that   $ e^{\delta  \eps <x>} f \in H^s$ equipped with the norm
\begin{equation}
\label{defwnorm}
\|f \|_{H^s_{\eps, \delta} } =   \eps^{\mez} \sum_{k \le s} \eps^{-k}  \|e^{\delta \eps <x>} \D_x^k f \|_{L^2}.
\end{equation}
 Note that for $\delta \le 1$, this norm is equivalent, with constants independent of $\eps$ and $\delta$, 
 to the norm
 $$
\|e^{\delta \eps <x>}  f \|_{H^s_\eps} . 
 $$

\begin{prop}\label{invprop}
Under the assumptions of Theorem \ref{main},  
there are  constants $C$,  $\eps_0 > 0 $ and $\delta_0 > 0$   
and  for all  $\eps \in ]0, \eps_0]$, there is a unit vector $\ell_\eps$ such that 
for $\eps \in ]0, \eps_0]$, $\delta \in [0, \delta_0]$, 
$f \in H^{3}_{\eps, \delta} $, $g \in H^{2}_{\eps, \delta} $ 
the operator equations \eqref{neweq} \eqref{phasecond} has a unique 
solution 
$ U \in H^{2}_{\eps, \delta} $, denoted by $ U = \csLd F$, which satisfies     
  \begin{equation}\label{invbdH2}
\big\|\csLd  F \big\|_{H^2_{\eps, \delta} }\le 
C\eps^{-1}\big( \big\| f \|_{H^{3}_{\eps, \delta} }
+ \big\|g  \big\|_{H^2_{\eps, \delta} }\big).
\end{equation}
 
 Moreover, for  $s \ge 3$, there is a constant $C_s$ such that for  
 $\eps \in ]0, \eps_0]$ 
and 
$f \in H^{s+1}_{\eps, \delta} $, $g \in H^{s}_{\eps, \delta} $ 
the  
solution 
$ U = \csLd F \in H^{s}_{\eps, \delta} $ and 
 \begin{equation}\label{invbdHs}
\big\|\csLd  F \big\|_{H^s_{\eps, \delta} }\le 
C\eps^{-1}\big( \big\| f \|_{H^{s+1}_{\eps, \delta} }
+ \big\|g  \big\|_{H^s_{\eps, \delta} }\big) + C_s  \big\|\csLd  F \big\|_{H^{s-1}_{\eps, \delta} } . 
\end{equation}

\end{prop}

The proof of this proposition  comprises most of the work of the paper.
Once it is established, existence follows by a straightforward
application of the Contraction-Mapping Theorem.
Defining 
\begin{equation}\label{Tdef}
\cT:=\csLd
\begin{pmatrix} 0 \\ - \cR_v + MU + N(U))\end{pmatrix},
\end{equation}
we reduce \eqref{pert} to the fixed-point equation
\begin{equation}\label{fixedeq}
\cT   U:=    U.
\end{equation}


\subsection{Proof of the basic result}\label{pf}

\begin{proof}[Proof of Theorem \ref{main}]

The profile $\bar u_{NS}$ exists if $\eps$ is small enough. 
The estimates \eqref{NSbds}  imply that 
\begin{equation}
\label{NSbds2}
\| \bar u_{NS} - u_\pm \|_{H^s_{\eps, \delta}}  \le  C_s  \eps 
\end{equation}
with $C_s$ independent of $\eps$ and $\delta$, provided 
that  $\delta \le \theta / 2$. 
Similarly, \eqref{eq:resbds} implies that
\begin{equation}\label{L2resbds}
\| \cR_v\|_{H^s_{\eps, \delta}}\le C_s  \eps^{  3}, 
\end{equation}
and \eqref{defM} implies that 
\begin{equation}\label{Mbds}
\|  M  \|_{H^s_{\eps, \delta}}\le C_s  \eps^{  2}.  
\end{equation}

Moreover, with the choice of norms \eqref{defnorm}, the Sobolev inequality  reads
\begin{equation}
\label{sobemb}
\|  u \|_{L^\infty }  \le  C   \| u \|_{H^1_{\eps}} \le C   \| u \|_{H^1_{\eps, \delta}}
\end{equation}
with $C $ independent of $\eps$. 
Moreover,  for smooth functions $\Phi$, there are nonlinear estimates 
\begin{equation}
\label{nlsest}
\| \Phi (u)   \|_{H^s_{\eps}} \le     C\big(  \|  u \|_{L^\infty } \big)    \ \| u \|_{H^s_{\eps}} . 
\end{equation}
which also extend to weighted spaces,  for $\delta \le 1$: 
\begin{equation}
\label{nlwsest}
\| \Phi (u)   \|_{H^s_{\eps, \delta}} \le     C\big(  \|  u \|_{L^\infty } \big)    \ \| u \|_{H^s_{\eps, \delta}} . 
\end{equation}

In particular, this implies that  for $s \ge 1$, $\delta \le \min\{ 1, \theta/ 2 \}$ and 
$\eps$ small  enough:  
\begin{equation}\label{Mbds2}
\begin{aligned}
\|  M  U   \|_{H^s_{\eps, \delta}}& \le C  \big(  \|  M  \|_{H^1_{\eps, \delta}}     \| U \|_{H^s_{\eps, \delta}}
+    \|  M  \|_{H^s_{\eps, \delta}}     \| U \|_{H^1_{\eps, \delta}}    \big) 
\\
& \le     \eps^{ 2 }   \big(  C     \| U \|_{H^s_{\eps,\delta}}   +  C_s     \| U \|_{H^1_{\eps,\delta}} \big)   
\end{aligned}
\end{equation}
where the first constant $C$ is independent of $s$. 
Similarly, 
 \begin{equation}\label{Nbds2}
\|  N (  U)    \|_{H^s_{\eps, \delta}}\le    C\big(  \|  U \|_{L^\infty } \big)      
 \| U \|_{H^1_{\eps, \delta}}  
 \| U \|_{H^s_{\eps, \delta}}  . 
\end{equation}

Combining these estimates,  we find that
\begin{equation*}
\begin{aligned}
\|\cT U\|_{H^s_{\eps, \delta} } \le   
 \eps^{-1} \big( C_s   \eps^{3 }  +   C  \eps^{2  }    \|U\|_{H^s_{\eps, \delta} } 
 + C_s  \eps^{2  }    \|U\|_{H^1_{\eps, \delta} }  
+  C  \|U\|_{H^1_{\eps, \delta} } \|U\|_{H^s_{\eps, \delta} }  \big), 
\end{aligned}
\end{equation*}
that is
\begin{equation}\label{Tbd}
\begin{aligned}
\|\cT U\|_{H^s_{\eps, \delta} } \le   
  C_s   \eps^{2 }  +   C  ( \eps^{  } +  \eps^{-1}\| U \|_{H^1_{\eps, \delta}} )   \|U\|_{H^s_{\eps, \delta} } 
 + C_s  \eps    \|U\|_{H^1_{\eps, \delta} }  
\end{aligned}
\end{equation}
provided that $\eps \le \eps_0$,  $\delta \le \min\{1, \theta / 2\}$ and 
$\| U \|_{L^\infty} \le 1$.

\medbreak

Consider first the case $s = 2$. 
Then,   $\cT$  maps the   ball 
$\cB_{\eps, \delta} = \{  \| U \|_{H^{2}_{\eps, \delta}}\le  \eps^{1+\frac{1}{2}}\} $ 
    to itself,  if  $\eps \le \eps_1 $ where  $\eps_1 > 0$ is small enough. 
  Similarly, 
  \begin{equation}\label{dTbd}
\begin{aligned}
\|\cT U -\cT V\|_{H^2_{\eps, \delta}}   \le 
C\eps^{-1} \big(\eps^{2} +  \|U\|_{H^2_\eps}  +  \|V\|_{H^2_\eps}\big)  \|U-V\|_{H^2_{\eps, \delta}},\\
\end{aligned}
\end{equation}
provided that  $\| U \|_{L^\infty} \le 1$ and $\| V \|_{L^\infty} \le 1$, 
from which we readily find   
that, for $\eps>0$ sufficiently small,
$\cT$ is contractive on   $\cB_{\eps, \delta}$, whence, by the Contraction-Mapping Theorem,
there exists a unique solution $  U^\eps$ of  \eqref{fixedeq} 
in $\cB_{\eps, \delta}$ for $\eps$ sufficiently small.

Moreover, from the contraction property 
$$
\|\bar U^\eps-\cT(0)\|_{H^2_\eps}= 
\|\cT(\bar U^\eps)-\cT(0)\|_{H^2_\eps}\le
c \|\bar U^\eps \|_{H^2_\eps},
$$
with $ c <1$, we obtain as usual that
$\|  U^{\eps, \delta} \|_{H^2_{\eps, \delta} }\le C\|\cT(0)\|_{H^2_{\eps, \delta}}$,
whence 
\begin{equation}
\|    U^\eps\|_{H^2_{\eps, \delta}  }\le C\eps^{2}. 
\end{equation}
by \eqref{Tbd}.
In particular,  $e^{\eps \delta  \la x \ra } U^\eps =  O( \eps^{2})$ in $H^2_{\eps}$ 
and by the Sobolev embedding 
\begin{equation}
\|  e^{\eps \delta  \la x \ra } U^\eps \|_{L^\infty}   =  O( \eps^{2}), 
\quad \|  e^{\eps \delta  \la x \ra } \D_x U^\eps \|_{L^\infty}   =  O( \eps^{3}). 
\end{equation}

\medbreak 

For $s \ge 3$, the estimates \eqref{Tbd} show that for $\eps \le \eps_1$ independent 
of $s$, the iterates $\cT^n (0)$ are bounded in $H^{s}_{\eps, \delta}$,   
and similarly that $\cT^n(0) - \cT (0) = O (\eps^2)$ in $H^{s}_{\eps, \delta}$, 
implying that the limit $U $ belongs to  $H^{s}_{\eps, \delta}$ with norm 
$O(\eps^2)$. 
Together with the Sobolev inequality \eqref{sobemb}, this implies the pointwise estimates  
\eqref{finalbds}.

Finally, the assertion about uniqueness follows by
uniqueness in $\cB_{c\eps, \delta}$ 
under the additional phase condition \eqref{phasecond} 
for the choice $\delta=0$
and $c>0$ sufficiently small
(noting by our argument that also $\cB_{c\eps, \delta}$ is mapped to 
itself for $\eps$ sufficiently small, for any $c>0$),
together with the observation that phase condition
\eqref{phasecond} may be achieved for any solution $\bar U=\bar U_{NS}+U$
with 
$$
\|U'\|_{L^\infty}\le c \eps^{2}<< \bar U_{NS}'(0)\sim \eps^2
$$
by translation in $x$, yielding
$\bar U_a(x):=\bar U(x+a)= \bar U_{NS}(x)+ U_a(x)$
with 
$$
U_a(x):= \bar U_{NS}(x+ a)-\bar U_{NS}(x)+ U(x+a) 
$$
so that
$\partial_a (\ell_\eps  \cdot u_a(0))=
\ell_\eps \cdot \big(\bar u_{NS}'(a) + u'(a)\big) \sim
\ell_\eps \cdot \bar u_{NS}'(0) $
and so (by the Implicit Function Theorem applied to $h(a):=\eps^{-2}
\big(\ell_\eps\cdot u_a)$, together with $\ell_\eps\cdot u_0=o(\eps)$
and the assumed property that $\ell_a\cdot \bar u_{NS}'(0)\sim \eps^{2}$
coming from our choice of $\ell_\eps$; see \eqref{normalization},
Remark \ref{ellchoice})
the inner product $\ell_a\cdot \bar u_{NS}'(0)$
may be set to zero by appropriate choice of $a=o(\eps^{-1})$ leaving
$U_a$ in the same $o(\eps)$ neighborhood, by the computation
$U_a-U_0\sim \partial_a U \cdot a\sim o(\eps^{-1})\eps^2$.
\end{proof}

It remains to prove existence of the linearized solution
operator and the linearized bounds \eqref{invbdHs}, which
tasks will be the work of most of the rest of the paper.
We concentrate first on estimates, and prove the existence next, using 
a viscosity method 
combined with (the single new step in treating the infinite-dimensional
case) discretization in velocity.

%

\section{Internal and high frequency estimates}
 \label{energy}

We begin by establishing a priori estimates on solutions
of the equation \eqref{neweq}
This will be done in two stages.
In the first stage, carried out in
this section, we establish energy estimates
showing that ``microscopic'', or ``internal'', variables consisting
of $v $ and derivatives of $(u, v)$ are controlled by  and small 
with respect to the ``macroscopic'', or ``fluid'' variable, $u$.
As discussed in Section \ref{basic}, this is the main new
aspect in the infinite-dimensional case.

In the second stage, carried out in Section \ref{linCEestimates}, we
estimate the macroscopic variable $u$ by Chapman--Enskog approximation
combined with finite-dimensional ODE techniques such as have been
used in the study of fluid-dynamical shocks
\cite{MaZ4, MaZ5, Z1, Z2, GMWZ}, exactly as in the finite-dimensional
analysis of \cite{MeZ1}.

\subsection{The basic $H^1$ estimate} 
  
  We consider the equation
  \begin{equation}
  \label{inteqs6}
  \cL_*^\eps U := \begin{pmatrix} A_{11} u + A_{12} v  
  \\
   A_{21} u' + A_{22} v'  -   dq (\bar u_{NS} , v_* (\bar u_{NS}))  U \end{pmatrix} = 
  \begin{pmatrix} f \\ g \end{pmatrix}  
  \end{equation}
   and its differentiated form: 
\begin{equation}\label{apriorieq}
AU'- dQ(\bar u_{NS}, v_*(\bar u_{NS}))U=
\begin{pmatrix} f'\\g \end{pmatrix}.
\end{equation}
The internal variables are $U' = (u', v')$ and $\tilde v$ where 
\begin{equation}
\label{tildev}
\tilde v:= 
v   +    p   u  , \qquad 
p =  \partial_v q^{-1}\partial_uq (\bar u_{NS}, v_*(\bar u_{NS})) =  
-  dv_* (\bar u_{NS})
\end{equation} 
is the linearization about $(\bar u_{NS},\bar v_{NS})$
of the key variable $v-v_*(u)$ arising in the Chapman--Enskog expansion
of Section \ref{CEapprox}.
Noting that $pu=0$ at the reference point $\underline U$ 
by Assumption \ref{refass}, we have the important fact that
\be\label{smallp}
\|pu\|_{\HH}=O(\eps) \|u\|_{\HH}
\ee
on the set of $U$ we consider ($\eps^2$ close to $\bar U_{NS}$, so
$\eps$ close to $\underline U$), so that $v$ and $\tilde v$ are
nearly equivalent.

\begin{prop}\label{energypropL2}
Under the assumptions of Theorem \ref{main}, for  there  are  constants 
$C$, $\eps_0 > 0$ and $\delta_0 > 0$ such that for  $0 < \eps \le \eps_0$ and 
$0 \le \delta \le \delta_0$, 
$f \in H^{2}_{\eps, \delta} $, $g \in H^{1}_{\eps, \delta} $ 
and      $U= (u,v)\in H^1_{\eps, \delta}$ of \eqref{inteqs6} satisfies
  \begin{equation}\label{invbd}
\big\| U'   \big\|_{L^2_{\eps, \delta} }   + \big\| \tilde v   \big\|_{L^2_{\eps, \delta} }  \le 
C    \big( \big\| (f, f', f'', g, g') \|_{L^2_{\eps, \delta} }
 + \eps  \big\| u  \big\|_{L^2_{\eps, \delta} }  \big).
\end{equation}

\end{prop}

\begin{proof}
For $\delta=0$, the result follows by Lemma \ref{basicest}
together with \eqref{smallp}.

For $\delta > 0$ small, consider 
$U ^w =  e^{ \eps \delta \la x \ra } U$. Then, $U^w$ satisfies 
 \begin{equation}
  \label{inteqs6w}
  \cL_*^\eps U^w  = 
  \begin{pmatrix} f^w \\ g^w    \end{pmatrix}  ,   
  \end{equation}
  with $f^w =   e^{ \eps \delta \la x \ra }  f $ and $g^w  =   e^{ \eps \delta \la x \ra }  g + \eps \delta   \la x\ra ' (A_{21} u^w + A_{22}  v^w ) $. 
  We note that, 
  $$
  \|  U'  \|_{L^2_{\eps, \delta}}  \le    \|  (U^w)'  \|_{L^2_{\eps}}  +   \eps \|   U^w  \|_{L^2_{\eps}} , 
, \quad 
    \|  \tilde v  \|_{L^2_{\eps, \delta}} \lesssim   \|  \tilde v^w  \|_{L^2_{\eps}} , 
  $$
  $$
    \|   f^w, (f^w)', (f^w)''   \|_{L^2_{\eps}} \lesssim   \|  (f, f', f'')  \|_{L^2_{\eps, \delta}} , 
$$
$$
\begin{aligned}
     \|   g^w, (g^w)'  \|_{L^2_{\eps}}   \lesssim   \|  (g, g') \|_{L^2_{\eps, \delta}}  + 
    \eps  \delta \| (U, U') \|_{L^2_{\eps, \delta}}.  
  \end{aligned}
  $$
   We use the estimate \eqref{invbd} with $\delta = 0$ for 
 $U^w$, and the Proposition follows provided that $\delta$ is small enough.
\end{proof}

\subsection{Higher order estimates} 

\begin{prop}\label{estHs}  There are constants $C$, $\eps_0 > 0$, $\delta_0 > 0$ and for 
all $k \ge 2$, there is  $C_k$, such that 
 $0< \eps \le \eps_0$, $\delta \le \delta_0$,  
 $U \in H^s_{\eps, \delta}$, 
  $f \in H^{s+1}_{\eps, \delta}$ and $g \in H^{s}_{\eps, \delta}$ 
  satisfying 
\eqref{apriorieq} satisfies:
\begin{equation}\label{sharph2eq}
\begin{aligned}
  \| \D_x^k U' \|_{L^2_{\eps, \delta}} + 
&  \| \partial^k_x \tilde v \|_{L^2_{\eps, \delta}}  
  \le  C  
  \| \partial_x^k (f, f', f'', g, g')  \|_{L^2_{\eps, \delta} }     \\
  &  +    \eps^k  C_k \big( 
 \| U' \|_{H^{k-1}_{\eps, \delta} } + \eps   \| \tilde v \|_{H^{k-1}_{\eps, \delta} }  +
  \eps \| u \|_{L^2_{\eps, \delta}}\big)  
\end{aligned}
\end{equation}

\end{prop}

\begin{proof} 
Differentiating \eqref{inteqs6} $k$ times, yields
\begin{equation}\label{Dapriorieq}
A\partial_xU^{k} - dQ(\bar u_{NS}, v_*(\bar u_{NS}))\partial_x U^{k} =
\begin{pmatrix} \partial^k_x f'    \\ \partial^{k}_xg +   r_k  \end{pmatrix},
\end{equation}
where 
$$
r_k  =  - \sum _{l=0}^{k-1} \partial^{k-l}_x  Q_{22} \,  \D_{x}^{l} \tilde v. 
$$
Here we have used that $d q (\bar u_{NS}, v_* (\bar u_{NS}) U = Q_{22} \tilde v$. 
The   $H^1$ estimate 
  yields
$$
\begin{aligned}
\| \D_x^k U' \|_{L^2_{\eps, \delta}} + 
 \| \partial^k_x  v + p \D_{x}^k u  \|_{L^2_{\eps, \delta}}   \le  
  C  \big( & \| \partial_x^k (f, f', f'', g, g')   \|_{L^2_{\eps, \delta}}     \\
     + \eps  \| \partial_x^k  u \|_{L^2_{\eps, \delta}}  
 & +   \| \partial_x r_k \|_{L^2_{\eps, \delta}} 
+   \| r_k  \|_{L^2_{\eps, \delta}} \big) , 
\end{aligned}
$$
for $0 \le k \le s$, with  $r_0 = 0$ when $k = 0$. 
Since $Q$ is a function of $\bar u_{NS}$,   its   $k- l$-th derivative   is $O (\eps^{k - l+1})$
when $k-l > 0$. 
Therefore: 
$$
 \| \partial_x r_k \|_{L^2_{\eps, \delta}} 
+   \| r_k  \|_{L^2_{\eps, \delta}} \le C_k  \eps^{k}  \big(  \|\tilde  v' \|_{H^{k-1}_{\eps, \delta}} +
\eps   \| \tilde v \|_{L^2_{\eps, \delta} } \big). 
$$
Similarly,  for $k = 1$
$$
 \| \partial_x \tilde v_k \|_{L^2_{\eps, \delta}} 
 \le  \| \partial_x  v +  p \D_{x} u  \|_{L^2_{\eps, \delta}}  +  C 
 \eps^2  \| u  \|_{L^2_{\eps, \delta}}   
$$
and for $k \ge 2$: 
$$
 \| \partial^k_x \tilde v_k \|_{L^2_{\eps, \delta}} 
 \le  \| \partial^k_x  v + p \D_{x}^k u  \|_{L^2_{\eps, \delta}}  +  C_k 
(   \eps^k \| u'   \|_{H^{k-2}_{\eps, \delta}} +  \eps^{k+1}    \|\tilde  u  \|_{L^2 _{\eps, \delta}} \big). 
$$

\end{proof}


\section{Linearized Chapman--Enskog estimate} \label{linCEestimates}

\subsection{The approximate equations}

It remains only to estimate $\|u\|_{L^2_{\eps, \delta}}$      in order to close the estimates
and establish \eqref{invbd}.
To this end, we work with the first equation  in 
\eqref{inteqs6}
and  estimate it by comparison with the Chapman-Enskog 
approximation (see the computations Section~\ref{CEapprox}),
exactly as in the finite-dimensional case \cite{MeZ1}. 
  
From the second equation
$$
A_{21}u'+A_{22}v' -g=\partial_u q u+ \partial_v q v= \partial_vq \tilde v,
$$
where we use the notations $\tilde v$ of Proposition~
\ref{energypropL2}, 
we find 
\begin{equation}
\label{T2s7}
\tilde v=\partial_v q^{-1}
\Big((A_{21} + A_{22}\partial_v dv_* (\bar u_{NS}))u ' +A_{22} \tilde v  ' -g
\Big). 
\end{equation}
Introducing $\tilde v$ in the first equation, yields 
$$
(A_{11} + A_{12} d v_* (\bar u_{NS} ) ) u  + A_{12} \tilde v =  f,  
$$ 
 thus
$$
(A_{11} + A_{12} dv_* (\bar u_{NS}) ) u' =  f' - A_{12} \tilde v' -
d^2 v_* (\bar u_{NS}) (\bar u'_{NS}, u) . 
$$
 Therefore, \eqref{T2s7} can be modified to 
 \begin{equation}
 \label{T2bs7}
 \tilde v  =  c_* (\bar u_{NS}) u'   +   r 
\end{equation}
 with 
 $$
 \begin{aligned}
 r = d^{-1}_vq (\bar u_{NS}, &v_* (\bar u_{NS})) \Big(  A_{22}(\tilde v)' -g
 \\
& + dv_* (\bar u_{NS}) \big(  f' - A_{12} \tilde v' -
d^2 v_* (\bar u_{NS}) (\bar u'_{NS}, u)\big) \Big) . 
 \end{aligned}
 $$
This implies that $u$ satisfies the linearized profile equation
\begin{equation}\label{intpertu}
\begin{aligned}
\bar b_* u'- \bar {dh}_* u  =   A_{12} r  - f 
\end{aligned}
\end{equation}
where $\bar b_*=b_*(\bar u_{NS})$ 
and $\bar {dh}_{*} := dh_*(\bar u_{NS}) = A_{11} + A_{12} dv_* (\bar u_{NS})$.


\subsection{$L^2$ estimates and proof   of the main estimates}

The following estimate was established in \cite{MeZ1}
using standard finite-dimensional ODE techniques; for completeness,
we recall the proof here as well, in Section \ref{recall} below.

\begin{prop}[\cite{MeZ1}]\label{uprop}
The operator $\bar b_*\partial_x -\bar{dh}_*$ has
a right inverse $(b_*\partial_x -dh^*)^{\dagger}$ 
satisfying
\begin{equation}\label{rightinv}
\|(\bar b_*\partial_x -\bar {dh}_*)^{\dagger}h\|_{L^2_{\eps, \delta}} \le 
C\eps^{-1}\|h\|_{L^2_{\eps, \delta}},
\end{equation}
uniquely specified by the property that the solution 
$u = (b_*\partial_x -dh^*)^{\dagger} h$  satisfies 
\begin{equation}\label{phase}  
\ell_\eps  \cdot u(0) =0. 
\end{equation}
for a certain unit vector $\ell_\eps$. 
\end{prop}

Taking this proposition for granted, we finish the proof of the main estimates in Proposition~\ref{invprop}.

 \begin{prop}
 \label{prop72}
There are  constants $C$,  $\eps_0 > 0 $ and  $\delta_0 > 0$   
such that 
for $\eps \in ]0, \eps_0]$, $\delta \in [0, \delta_0]$,       
$f \in H^{3}_{\eps, \delta} $, $g \in H^{2}_{\eps, \delta} $ and 
$U \in H^2_{\eps, \delta}$ satisfying  
  \eqref{neweq} and \eqref{phasecond}  
  \begin{equation}
  \label{invbdH2s7}
\big\| U \big\|_{H^2_{\eps, \delta} }\le 
C\eps^{-1}\big( \big\| f \|_{H^{3}_{\eps, \delta} }
+ \big\|g  \big\|_{H^2_{\eps, \delta} }\big).
\end{equation}
\end{prop}

\begin{proof}
Going back now to \eqref{intpertu}, $u$ satisfies 
$$
\begin{aligned}
\bar b_* u'- \bar {dh}_* u &=  O(|\tilde v'|+ |g| + |f'| + \eps^2 | u |  )  -   f,
\end{aligned}
$$
   If in addition  $u$ satisfies the condition \eqref{phase}
then  
\begin{equation}
\label{temp2}
\|u\|_{L^2_{\eps, \delta}}\le C   \eps^{-1} 
( \|\tilde v'\|_{L^2_{\eps, \delta} }
+ \|(f, f',g)\|_{L^2_{\eps, \delta} }     + \eps^2  \| u \|_{L^2_{\eps, \delta}} \big) . 
\end{equation}

By  Proposition~\ref{energypropL2} and Proposition~\ref{estHs} for $k = 1$, we have 
  \begin{equation} 
  \label{est77}
\big\| U'   \big\|_{L^2_{\eps, \delta} }   + \big\| \tilde v   \big\|_{L^2_{\eps, \delta} }  \le 
C    \big( \big\| (f, f', f'', g, g') \|_{L^2_{\eps, \delta} }
 + \eps  \big\| u  \big\|_{L^2_{\eps, \delta} }  \big).
\end{equation}
\begin{equation}
\label{est78}
\begin{aligned}
  \|   U'' \|_{L^2_{\eps, \delta}} + &
   \big\| \tilde v '   \big\|_{L^2_{\eps, \delta} }  \le 
   \\
& C    \big( \big\| (f', f'', f''',  g',   g'') \|_{L^2_{\eps, \delta} }
 + \eps  \big\| U'  \big\|_{L^2_{\eps, \delta} }   + 
 \eps^2   \big\| u  \big\|_{L^2_{\eps, \delta} } \big).
\end{aligned}
 \end{equation}
Combining these estimates,  
this implies 
 \begin{equation*} 
 \begin{aligned}
  \big\| \tilde v '   \big\|_{L^2_{\eps, \delta} }  & \le    
C      \big( \big\| (f', f'', f''',  g',   g'') \|_{L^2_{\eps, \delta} } + \eps \big\| (f, f', f'', g, g') \|_{L^2_{\eps, \delta} }
 + \eps^2  \big\| u  \big\|_{L^2_{\eps, \delta} }  \big)\\
& \le    
C      \big(   \eps \big\| (f, f', f'', g, g') \|_{H^1_{\eps, \delta} }
 + \eps^2  \big\| u  \big\|_{L^2_{\eps, \delta} }  \big).
 \end{aligned}
\end{equation*}
Substituting in \eqref{temp2}, yields 
$$ 
\eps \|u\|_{L^2_{\eps, \delta}} \le C \big(  \|(f, f',g )\|_{L^2_{\eps, \delta}} + 
 \eps \|(f, f',f'',g, g' )\|_{H^1_{\eps, \delta}}    
+ \eps^2 \|  u \|_{L^2_{\eps, \delta}} \big). 
$$
Hence for $\eps $ small,  
\begin{equation}\label{temp3}
\eps \|u\|_{L^2_{\eps, \delta}} \le C \big(  \|(f, f',g )\|_{L^2_{\eps, \delta}} + 
 \eps \|(f, f',f'',g, g' )\|_{H^1_{\eps, \delta}}     \big). 
\end{equation}
 
Plugging this estimate in \eqref{est77} 
 \begin{equation} 
  \label{est711}
\big\| U'   \big\|_{L^2_{\eps, \delta} }   + \big\| \tilde v   \big\|_{L^2_{\eps, \delta} } 
+  \eps  \big\| u  \big\|_{L^2_{\eps, \delta} }  \le 
C     \big\| (f, f', f'', g, g') \|_{H^1_{\eps, \delta} }
 +  \big).
\end{equation}
Hence, with \eqref{est78}, one has 
\begin{equation}
\label{est712}
\begin{aligned}
  \|   U'' \|_{L^2_{\eps, \delta}} + &
   \big\| \tilde v '   \big\|_{L^2_{\eps, \delta} }  \le 
   \\
& C    \big( \big\| (f', f'', f''',  g',   g'') \|_{L^2_{\eps, \delta} }
 + \eps   \big\| (f, f', f'', g, g') \|_{H^1_{\eps, \delta} }\big).
\end{aligned}
 \end{equation}
Therefore, 
 \begin{equation}
\label{est713}
 \big\| U'   \big\|_{H^1_{\eps, \delta} }          
+  \big\| \tilde v   \big\|_{L^2_{\eps, \delta} } 
+  \eps  \big\| u   \big\|_{L^2_{\eps, \delta} }  \le 
  C     \big\| f, f', f'', g, g'  \big\|_{H^1_{\eps, \delta} }  
\end{equation}
The left hand side dominates 
$$
  \big\| U'   \big\|_{H^1_{\eps, \delta} }  + \eps  \big\| U'   \big\|_{L^2_{\eps, \delta} } 
  =  \eps \big\| U'   \big\|_{H^2_{\eps, \delta} } 
  $$
and the right hand side is smaller than or equal to  
$  \big\|  f  \big\|_{H^2_{\eps, \delta} } +  \big\| g   \big\|_{H^1_{\eps, \delta} } $.
The estimate \eqref{invbdH2s7}  follows.   
\end{proof}

Knowing a bound  for $\| u \|_{L^2_{\eps, \delta}}$, Proposition~\ref{estHs} immediately implies

\begin{prop}\label{prop73}
There are  constants $C$,  $\eps_0 > 0 $ and  $\delta_0 > 0$   
and  for  $s \ge 3$  there is a constant $C_s$
such that 
for $\eps \in ]0, \eps_0]$, $\delta \in [0, \delta_0]$,     
 $f \in H^{s+1}_{\eps, \delta} $, $g \in H^{s}_{\eps, \delta} $ and 
$U \in H^s_{\eps, \delta}$ satisfying  
  \eqref{neweq} and \eqref{phasecond}, one has 
 \begin{equation}\label{invbdHs7}
\big\|  U   \big\|_{H^s_{\eps, \delta} }\le 
C\eps^{-1}\big( \big\| f \|_{H^{s+1}_{\eps, \delta} }
+ \big\|g  \big\|_{H^s_{\eps, \delta} }\big) + C_s  \big\|  U  \big\|_{H^{s-1}_{\eps, \delta} } . 
\end{equation}

\end{prop}

\begin{rem}\label{goodman}
\textup{
The estimate of Proposition \ref{uprop} may be recognized
as somewhat similar to the estimates of
Goodman \cite{Go} obtained by energy methods in the time-evolutionary case,
the same ones used by Liu and Yu \cite{LY} to control the macroscopic
variable $u$.
More precisely, the argument is a simplified version of the 
one used by Plaza and Zumbrun \cite{PZ} to show time-evolutionary
stability of general small-amplitude waves.
}
\end{rem}


\subsection{Proof of Proposition~\ref{uprop}}\label{recall}
 
By Assumption \ref{goodred}(i), we may assume that there are  linear coordinates  
$u = (u_1, u_2) \in \RR^{n_1} \times \RR^{n_2}$ 
and $h =( h_1, h_2)  \in \RR^{n_1} \times \RR^{n_2}$, with $n_2 = \mathrm{rank} \ b_* (\bar u) $  
 such that   
\begin{equation}
\label{blokbs}
   b_* (\bar u)  = \begin{pmatrix}0 & 0 \\ b_{21} (\bar u) & b_{22} (\bar u) \end{pmatrix}
 \end{equation}
 and $b_{22}(\bar u)$ is uniformly invertible on   $ \cU_*$. 
 Introducing the new variable
 \begin{equation}
 \label{goodu}
 \tilde u_2 = u_2 +   \bar V  u_1, 
 \quad \bar V =  ( b^{22}) ^{-1}  b_{21}   (\bar u_{NS}),
 \end{equation}
 the equation 
$\bar b_* u'- \bar{dh}_*u =h$  has the form: 
 \begin{equation}
 \label{blockprofeq}
 \begin{aligned}
\bar a^{11}  u_1 + \bar a^{12}  \tilde u_2 =  h_1,\\
\bar b^{22}  \tilde u_2'   - \bar a^{21}u_1 - \bar a^ {22} \tilde u_2  =     h_2
\end{aligned}
 \end{equation}
 where 
 $$
 \bar a :=   \bar {dh}_*     \begin{pmatrix} \Id & 0 \\ - \bar V   & \Id \end{pmatrix} 
 +   \bar b*  \begin{pmatrix} 0 & 0 \\   \bar V '   & 0 \end{pmatrix}. 
 $$
 
 Assumption \ref{goodred}(ii) implies that the left upper corner block 
 $\bar a^{11}$  is uniformly invertible. Solving the first equation
for $u_1$,   we obtain
the reduced nondegenerate ordinary differential equation 
$$
\bar b_*^{22} \tilde u_2'  +   \bar a^ {21}(\bar a^{11})^{-1}
  \bar a^ {12}  \tilde u_2 
- \bar a^{22} \tilde u_2=  h_2  +  \bar a^ {21}(\bar a^{11})^{-1}
h_1
$$
or
\begin{equation}\label{princ}
\begin{aligned}
\check b u_2'-\check a u_2
 = \check h =O(|h_1|+|h_2|). 
\end{aligned}
\end{equation}
 
Note that $\det \bar dh_* =\det \bar a^{11} \det \check a$ by
standard block determinant identities, 
so that $\det \check a  \sim \det \bar dh_*$ by Assumption \ref{goodred}(ii).
Moreover, as established in \cite{MaZ4},  by  Assumption \ref{profass} and the construction of the profile
$\bar u_{NS}$  
we find that $ m := (\check b)^{-1}\check a$ has the following properties:

 \quad  i)  with $m_\pm$ denoting the end  points values of $m$, there is $\theta > 0$ such that 
  for all $k$ : 
\begin{equation}
\label{est718}
| \D_x^k ( m(x) - m_\pm)  | \lesssim  \eps^{k+1} e^{ - \eps  \theta | x |};       
\end{equation}
   
\quad    ii)   $m (x)$  has a single simple eigenvalue   of order $\eps$, 
dented by $ \eps \mu (x) $,  and  there is $c > 0$ such that 
   for all $x$ and $\eps $ 
the other eigenvalues $\lambda$ satisfy  $| \Re \lambda | \ge c$; 

\quad iii)  the end point values  $\mu_\pm$ of $\mu$  satisfy
\begin{equation}
\label{est719} 
\mu_-  \ge   \alpha    \qquad \mu_+ \le - \alpha  
\end{equation}
for some $\alpha  > 0$ independent of $\eps$.

\smallbreak 
In the strictly parabolic case $\det b_*\ne 0$, this follows
by a lemma of Majda and Pego \cite{MP}.  

 At this point, we have reduced to the case 
\begin{equation}\label{semifinalform}
u_2'- m (x)   u_2 =  O(|h_1|+|h_2|),
 \end{equation}
with $m$ having the properties listed above. 
The important feature is that $m' = O (\eps^2)  << \eps $,  the spectral gap between stable, unstable, and $\eps$-order subspaces
of $m$.  The conditions above imply that there is a matrix 
$ \omega $ such that 
$$
p := \omega ^{ -1} m \omega   =    \blockdiag\{p^+,  \eps  \mu , p-\},
$$
 where the spectrum of $p_\pm$ lies in $ \pm \Re \lambda \ge c $. Moreover, $\omega $ 
 and $p$ satisfies estimates 
 similar to \eqref{est718}. The change of variables 
 $u_2 =  \omega  z$  reduces \eqref{semifinalform}
 to 
 \begin{equation}
 \label{finalform}
 z'  -   p z   =      \omega^{-1} \omega' z    +     O(|h_1|+|h_2|)  . 
 \end{equation}

The equations  $(z^+)'- p^+ z^+ = h^+ $ and
$(z^-)'- p^-z^-= h^- $ either by standard linear theory 
\cite{He} or by symmetrizer estimates as in \cite{GMWZ},  admit unique
  solutions in weighted $L^2$ spaces,  satisfying 
$$
\| e^{  \delta | x |}  z^ \pm\|_{L^2}\le C \| e^{   \delta | x |}  h^\pm \|_{L^2},  
$$
 provided that $\delta$ remains small, typically  $\delta < | \Re p^\pm |$. 

The equation $z_0' -  \eps \mu  z_0  =  h_0 $ may be converted by
the change of coordinates $x\to \tilde x:= \eps x$ to 
\begin{equation}\label{finalz0}
\D_{\tilde x}  \tilde z_0-\tilde  \mu (\tilde x) z_0=   \tilde h_0 (\tilde x)  = \eps^{-1} h_0 (\tilde x/ \eps) ,
\end{equation}
where $\tilde z_0 (\tilde x) = z_0(\tilde x / \eps) $ and $\tilde \mu (\tilde x):=
\mu (\tilde x/\eps)$. By  \eqref{est718} 
$$
| \tilde \mu (\tilde x)- \mu_\pm|\le Ce^{-\theta |\tilde x|}
$$
with $\mu_\pm $ satisfying  \eqref{est719}. 
This equation is underdetermined with index one, reflecting the 
translation-invariance of the underlying equations.
However, the operator $\partial_{\tilde x}-\tilde  \mu$ has a bounded
$L^2$ right inverse $(\partial_{\tilde x}-\tilde  \mu )^{-1}$, as
may be seen by adjoining an additional artificial constraint
\begin{equation}
  \tilde z_0(0) = 0
\end{equation}
 fixing the phase.  This can be seen by solving explicitly the 
 equation  or  applying the gap lemma
of \cite{MeZ2} to reduce the problem to two constant-coefficient
equations on $\tilde x\gtrless 0$, with boundary conditions 
at $z = 0$. 
 We obtain as a result that  
$$
\|e^{   \delta | \tilde x |} \tilde z_0\|_{L^2 }\le C  \| e^{   \delta | \tilde x |} \tilde h_0\|_{L^2}
$$
if $\delta < \min \{\alpha, \theta\}$, which
yields  by  rescaling  
the   estimate
$$
\|e^{ \eps  \delta | x |} z_0\|_{L^2 }\le C\eps^{-1} \| e^{ \eps  \delta | x |} h_0\|_{L^2}
$$

Together with the (better) previous estimates, this gives
existence and uniqueness for the equation
$$
z'  -p z = h , \qquad z_0(0)= 0
$$ with the estimate 
$\| e^{ \eps  \delta | x |}z\|_{L^2}\le C\eps^{-1} \| e^{ \eps  \delta | x |}h\|_{L^2}$.  Because $\omega^{-1} \omega' = O(\eps^2)$, 
this implies that for $\eps$ small enough, the equation \eqref{finalform} with  $z_0(0) = 0$ 
has a unique solution.   Tracing back to the original variables $u$, the condition
$z_0 (0) = 0$ translates into a condition of the form 
$\ell_\eps \cdot u(0) = 0$. Therefore,   
the equation 
$ \bar b_* u' -  \bar d f_* u = h $ has a unique solution such $u$ that  
$ 
\ell_\eps \cdot u(0) = 0
$, which satisfies 
$$
\|e^{ \eps  \delta | x |}u\|_{L^2}\le C\eps^{-1} \|e^{ \eps  \delta | x |} h\|_{L^2}
$$ 
for $\delta$ and $\eps $ small enough, finishing the proof of Proposition~\ref{uprop}.

 %
 %
 %
 
\section{Existence for the linearized problem}

The desired estimates \eqref{invbdH2} and \eqref{invbdHs} are given by 
Propositions~\ref{prop72} and \ref{prop73}. 
It remains to prove existence for the linearized problem
with phase condition $u(0)\cdot r(\eps)=0$.
This we carry out using a vanishing
viscosity argument.

Fixing $\eps$, consider in place of $\csL U=F$ the 
family of modified equations
\begin{equation}\label{modeq}
\csLe U:=\csL U - \eta \begin{pmatrix}u'\\v''\end{pmatrix}= F := 
\begin{pmatrix}f \\g\end{pmatrix},
\quad
\ell_\eps \cdot u(0) =0.
\end{equation}
Differentiating the first equation yields 
\begin{equation}\label{modeqd}
A  U'  - d Q (x) U  -   U'' = 
\begin{pmatrix}f' \\g\end{pmatrix},
\quad
\ell_\eps \cdot u(0) =0.
\end{equation}
where  $d Q (x)$ denotes here the matrix $ dQ (\bar u_{NS}, v_* (\bar u_{NS}))$.

\subsection{Uniform estimates} 
We first prove uniform a-priori estimates. 
We denote by   $ \sS  $ the   Schwartz space and for 
$\delta \ge 0$, by $\sS_{\eps \delta} $ the space of functions $u$ such that 
$e^{ \eps \delta \la x \ra } u \in \sS$, with $\la x \ra = \sqrt{1 + x^2}$ as in \eqref{modx}. 

\begin{prop}
\label{lemunifbds}  
There are  constants  $\eps_0 > 0 $, $\delta_0 > 0$   and $\eta_0> 0$, and for 
all $s \ge 2$ a constant $C_s$, 
such that 
for $\eps \in ]0, \eps_0]$, $\delta \in [0, \delta_0]$, $\eta \in ]0, \eta_0]$, 
and  $U  $ and $F$   in  $ \sS_{\eps\delta}(\RR) $, 
satisfying \eqref{modeq}  
  \begin{equation}\label{unifbds83}
\big\|  U  \big\|_{H^s_{\eps, \delta} }\le 
C_s \eps^{-1}\big( \big\| f \|_{H^{s+1}_{\eps, \delta} }
+ \big\|g  \big\|_{H^s_{\eps, \delta} }\big).
\end{equation}

\end{prop} 

\begin{proof}
 The argument of Proposition~\ref{energypropL2}
  goes through essentially unchanged, with
new $\eta$ terms providing additional favorable higher-derivative
terms sufficient to absorb new higher-derivative errors coming
from the Kawashima part.    

Thus we are led to equations of the form 
\eqref{apriorieq} with the additional
term $- \eta U''$ in the left hand side.  Using the symmetrizer $\cS $ \eqref{def615}, 
one gains 
$  \eta   \| U'' \|^2_{L^2}  +\lambda \| U' \|^2_{L^2} $ in the minorization 
of $\Re (\cS F, U ) $  and loses commutator terms which are dominated by 
$$
\eta \| S'' \|_{L^\infty} ( \| U'\|^2_{L^2}  + \| U \|_{L^2} \| U' \|_{L^2} ) 
+ \eta \| K \|_{L^\infty}  (  \| U' \|_{L^2} + \| U \|_{L^2}) \| U'' \|_{L^2} ,  
$$
which can be  absorbed by the left hand side  yielding uniform estimates  
\begin{equation}\label{sharph1eq8}
\sqrt \eta \| \widetilde U'' \|_{L^2}   +   \| \widetilde U'  \|_{L^2} + \| \tilde v \|_{L^2}  \le C \big(     
\|  f  \|_{H^2}   +   \| h \|_{H^1} + \| \tilde g  \|_{H^1}   
+ \eps \| u \|_{L^2} \big). 
\end{equation}
Going back to \eqref{modeqd}, this implies uniform estimates of the form 
  \begin{equation}\label{invbd6}
\sqrt \eta \| U''|_{L^2_{\eps, \delta}}  +   \big\| U'   \big\|_{L^2_{\eps, \delta} }   + \big\| \tilde v   \big\|_{L^2_{\eps, \delta} }  \le 
C    \big( \big\| (f, f', f'', g, g') \|_{L^2_{\eps, \delta} }
 + \eps  \big\| u  \big\|_{L^2_{\eps, \delta} }  \big).
\end{equation}
for $\delta = 0$, and next for $\delta \in [0, \delta_0]$ with 
$\delta_0 > 0$ small, as in the proof of Proposition~\ref{energypropL2}.  

\medbreak

When commuting derivatives to the equation, the additional term $ \eta \D_x^2$ brings no new 
term and the proof of Proposition~\ref{estHs} can be repeated without changes, yielding 
estimates of the form 
\begin{equation}\label{est87}
\begin{aligned}
\sqrt \eta \| D_x^k U'' \|_{L^2_{\eps, \delta} }   +  & \| \D_x^k U' \|_{L^2_{\eps, \delta}}   + 
   \| \partial^k_x \tilde v \|_{L^2_{\eps, \delta}}  
\\  
  \le  C  &
  \| \partial_x^k (f, f', f'', g, g')  \|_{L^2_{\eps, \delta} }     
 \\
 &  +    \eps^k  C_k \big( 
 \| U' \|_{H^{k-1}_{\eps, \delta} } + \eps   \| \tilde v \|_{H^{k-1}_{\eps, \delta} }  +
  \eps \| u \|_{L^2_{\eps, \delta}}\big)  . 
\end{aligned}
\end{equation}

\medbreak

Next, applying the Chapman--Enskog argument of
Section \ref{linCEestimates} to the viscous system, we obtain 
in place of \eqref{intpertu} the equation
\begin{equation}\label{viscintpertu}
\begin{aligned}
\bar b_* u'- \bar {dh}_* u &=   f +   O(|\tilde v'|+ |g| + |f'| ) 
+\eps^2 O( | u | )    +    \eta O(|u'|+|U''|),
\end{aligned}
\end{equation}
where the final $\eta$ term coming from artificial viscosity 
is treated as a source. One applies  Proposition~\ref{uprop} 
to estimate $ \eps \|u\|_{L^2_{\eps, \delta}}$  by the 
$ L^2_{\eps, \delta}$-norm of the right hand side, 
and continuing as in the proof of Proposition~\ref{prop72}, 
the estimate \eqref{est713} is now replaced by 
 \begin{equation}
\label{est89}
\begin{aligned}
\sqrt \eta    \| U'''\|_{L^2_{\eps, \delta}} &+  \big\| U'   \big\|_{H^1_{\eps, \delta} }          
+  \big\| \tilde v   \big\|_{L^2_{\eps, \delta} } 
+  \eps  \big\| u   \big\|_{L^2_{\eps, \delta} } 
\\
& \le 
  C    \big(   \big\| f, f', f'', g, g'  \big\|_{H^1_{\eps, \delta} }  
  + \eta  ( \| U'\|_{L^2_{\eps, \delta}}  + \| U''\|_{L^2_{\eps, \delta}} ) \big) . 
  \end{aligned}
\end{equation}
  Therefore, for $\eta$ small, the new  $O(\eta)$  terms can be absorbed, and 
  \eqref{unifbds83} for $s = 2$ follows as before. The higher order estimates 
 follow from \eqref{est87}. 
\end{proof} 


\subsection{Existence} 

We now prove existence and uniqueness for \eqref{modeq}. 
First, recast the the problem as a first-order system
\begin{equation}
\label{o1red}
\cU'  - \mA \cU = \cF 
\end{equation}
with 
$$
\cU =  \begin{pmatrix}
u\\v\\v'
\end{pmatrix}'
 , \qquad 
\cF = \begin{pmatrix} f\\0\\g \end{pmatrix},
$$
and 
\begin{equation}\label{AA}
\mA:=
\eta^{-1}
\begin{pmatrix}
 A_{11}& A_{12} & 0\\
0 & 0 & \eta I\\
\eta^{-1} A_{21}A_{11}-   Q_{21} & \eta^{-1} A_{21} A_{12}  -  Q_{22}  & A_{22}\\
\end{pmatrix}.
\end{equation}
Next, consider this as a transmission problem or a doubled boundary value problem on $x\gtrless 0$,
with boundary condtitions given by the $n+2r$ matching conditions $\cU(0^-)=\cU(0^+)$ at $x=0$  together with   the phase condition $\ell_\eps  \cdot u(0)  =0$,  
that is  $n+2r+1$ conditions in all: 
\begin{equation}
\label{transmcond}
\cU(0^-)=\cU(0^+), \qquad  \ell_\eps  \cdot u(0)  =0. 
\end{equation}

Note  that the operator-valued coefficient matrix $\mA$ 
converges exponentially to its endstates at $\pm\infty$,
by exponential convergence of $\bar U_{NS}$ and boundedness
of $A$, $Q$.

\begin{lem}
\label{lem82}  There is  $\theta_1 > 0$ such that for $\eps$ small enough , the limiting coefficient matrices 
$\mA_\pm$ have no eigenvalue in the strip 
$| \Re z  |  \le \eps \delta_0$. 
\end{lem}

\begin{proof}
The proof is parallel to the proof of the estimates. Dropping the $\pm$, 
 suppose  that  $i\tau$ is an eigenvalue of $\mA$, or equivalently that 
 there is a constant vector $   U\ne 0$ such that 
  $ e^{i\tau x}  U$ is a solution of 
of equations  \eqref{modeq}
Thus 
\begin{equation}
\label{eq814}
\begin{aligned}
& A_{11}   u + A_{12}   v = i \tau \eta   u,
\\
&(i \tau   A - Q  + \tau^2 \eta )   U = 0 . 
\end{aligned}
\end{equation}

In the first equation, introduce once again the variable 
$\tilde v = v + Q_{22}^{-1}  Q_{21 } u $, so that the equations 
are transformed to 
\begin{equation}
\label{eq815}
\begin{aligned}
& A^*_{11}  u + A_{12} \tilde v = i \tau \eta   u,
\\
&(i \tau A - Q + \tau^2 \eta )  U = 0 . 
\end{aligned}
\end{equation}
Denoting by $ K$ the end point values 
of the Kawashima multipliers associated  to  $A$ and 
$Q$,  consider the symmetrizer
$$
\Sigma = | \tau|^2   - i \overline \tau K  - \lambda  . 
$$
Multiplying the second equation in \eqref{eq815} by  $\Sigma$ and taking the real part of the 
scalar product with $U$ yields
$$
\begin{aligned}
  | \tau |^2   \Re (K  A & -  Q U, U)  
+ \lambda (Q U, U) + \eta | \tau |^4 (U, U) 
\\
&\le  C \big( | \im \tau | (| \tau |^2 + \lambda ) \big) | U |^2  +  C  
| \tau |  | Q U| | U | 
\\
& \qquad \qquad  + \eta ( | \tau |^2 |\im \tau|^2  + | \tau|^3 + \lambda | \tau |^2) \big)  | U |^2  . 
\end{aligned}
$$ 
Therefore, choosing appropriately $\lambda$, for $\eta$ and 
$|\im \tau |$ sufficiently small, one has 
\begin{equation}
\label{eq816}
(\eta | \tau |^4 +  | \tau |^2)  | U |^2 + | v |^2 \le C 
(| \im \tau |+\eps) | u |^2  
\end{equation}
In particular, $|\tau| $ must be small if $\im \tau $, $\eps$ are small. 

From the equation 
$i \tau A_{21} u + A_{22} v - 
Q_{21} u - Q_{22} v + \eta \tau^2 v  = 0 $
and the fact that $|Q_{21}|=O(\eps)$ by \eqref{smallp},
one deduces that 
$$
\tilde v -   i \tau (Q_{22})^{-1} A_{21} u  
=  O (| \tau |  + \eta | \tau |^2 ) | \tilde v | . 
$$
Substituting in the first equation of \eqref{eq815}, we obtain the 
Chapman-Enskog approximation 
\begin{equation*}
( A_{11}^*  - i \tau \bar b_* ) u =   O( \eta | \tau |  + 
  | \tau| + \eta | \tau|^2 )  | \im \tau|^\mez )  ) | u |
\end{equation*}
 where $\bar b_*$ denotes the end point value of the function  \eqref{bstar}.  
 Therefore,    
\begin{equation}
\label{eq817}
| ( \bar b_*)^{-1}  A^*_{11} u  - i \tau u  |  \le  C  |\im \tau|^\mez  |  \tau |  | u | 
\end{equation}
 with arbitrarily small $c >0$. 
 We know from Assumption~\ref{profass} that for 
 $\eps$ small, $( \bar b_*)^{-1}  A^*_{11}$ has a unique small eigenvalue, of order 
 $O(\eps)$, real. Let us denote it by $\eps \mu$. Then we know that 
 $| \mu |$ is bounded from below, see \eqref{est719}. Then 
 \eqref{eq817} implies that there is a constant $C$ such that  for $| \im \tau |$  small enough, 
 and thus $| \tau| $ small,   
 $| i \tau - \eps \mu | \le   C  |\im \tau|^\mez  |  \tau |  $. Therefore,  
 $ l \im \tau + \eps \mu |  \le \mez \eps | \mu |$ if $\eps$ is small enough. 
  
 Summing up, we have proved that if 
 $\eps$ is small enough, $\mA$ has at most one eigenvalue $z $   in the strip 
 $| \re z  \le   \eps 2 | \mu| $,  such that 
 $ |z - \eps  \mu |  \le \mez \eps | \mu|$. 
  This implies the lemma.
 \end{proof}

\begin{rem}
\textup{The same reasoning can be applied to prove that $\mA$ actually has a simple eigenvalue such that  $ |z - \eps  \mu |  \le \mez \eps | \mu|$. }
\end{rem}

\subsubsection{Finite-dimensional case}

We first review the case that $U$ is finite-dimensional, recalling
for completeness the analysis of \cite{MeZ1}.

\begin{prop}[\cite{MeZ1}]\label{viscexist}
There are  constants  $\eps_0 > 0 $, $\delta_0 > 0$   and $\eta_0> 0$ 
such that 
for $\eps \in ]0, \eps_0]$, $\delta \in [0, \delta_0]$, $\eta \in ]0, \eta_0]$, 
and $F$   in  $ \sS_{\eps\delta}(\RR) $, \eqref{modeq}
admits a unique solution $ U\in \sS_{\eps \delta} (\RR)$. 
\end{prop}

\begin{proof}
 Noting that the coefficient matrix $\mA$ converges exponentially to
 $\mA_\pm$ at   $\pm\infty$,  we may apply the conjugation lemma
of \cite{MeZ1} to convert the equation \eqref{o1red}  by an asymptotically trivial 
change of coordinates $\cU=T(x)Z$ to a constant-coefficient problems
\begin{equation}\label{ccprob}
Z_-' - \mA_- Z_-=F_-,
\quad
Z_+'-\mA_+ Z_+=F_+,
\end{equation}
on $\{ \pm x \ge 0 \}$, 
with $n+2r+1$ modified boundary conditions determined by the value of the
transformation $T$ at $x=0$, where $\mA_\pm:= \mA(\pm \infty)$, and
$Z_\pm(x):= Z(  x)$ for $ \pm x>0$.

By standard boundary-value theory (see, e.g., \cite{He}), to prove existence 
and uniqueness in the Schwartz space for the   problem \eqref{o1red} on 
$\{ x < 0\}$ and $\{ x > 0 \}$ with transmission conditions \eqref{transmcond},  it is sufficient
to show that 

\quad (i) the limiting coefficient matrices $\mA_\pm$ are hyperbolic,
i.e., have no pure imaginary eigenvalues, 

\quad (ii) the number of boundary conditions is equal to the number
of stable (i.e., negative real part)
eigenvalues of $\mA_+$ plus the number of unstable eigenvalues (i.e., positive real part) of
$\mA_-$, and

\quad  (iii) there exists no nontrivial solution of the
homogeneous equation $f=0$, $g=0$.

Moreover, since the eigenvalues of $\mA_\pm$  are located in  $\{ | \re z | \ge \theta_1 \eps$, 
the conjugated form \eqref{ccprob} of the equation show that if the source term $f$ has an exponential decay  $e^{ - \eps \delta \la x \ra }$ at infinity, then the bounded solution also 
has the same exponential decay, provided that 
$\delta < \theta_1$ . Therefore, the three conditions above are also sufficient to prove existence 
and uniqueness in $\sS_{\eps \delta} $ if $\eps$ and $\delta $ are small. 

\medbreak

Note that  (i) is a consequence of Lemma~\ref{lem82}, while (iii) follows from the 
estimate \eqref{unifbds83}. 
 To verify (ii), it is enough to establish the formulae
\begin{equation}\label{dims}
\begin{aligned}
\dim \cS(\mA_\pm)&= r+ \dim \cS(A_{11}^{*\pm}),\\
\dim \cU(\mA_\pm)&= r+ \dim \cU(A_{11}^{*\pm}),
\end{aligned}
\end{equation}
where $A_{11}^{* \pm}  = dh_* (u_\pm) = A_{11} + A_{12} dv_*(u_\pm)$ 
and $\cS(M)$ and $\cU(M)$ denote the stable and unstable
subspaces of a matrix $M$. 
We note that $A_{11}^{*\pm} =dh_*(u_\pm )$ are invertible, 
with dimensions of the  stable subspace of $A_{11}^{*+}$
and the unstable subspace of $A_{11}^{*-}$ summing to 
$n+1$, by Proposition \ref{NSprofbds}.
Thus, \eqref{dims} implies that 
$$
\dim \cS(\mA_+)+ \dim \cU(\mA_-)= 2r+ \dim \cS(A_{11}^{*+})
+ \dim \cU(A_{11}^{*-})= 2r + n+1
$$
as claimed.

To establish \eqref{dims}, introduce the variable 
$\tilde v = v + Q_{22}^{-1}  Q_{21 } u $, and the variable corresponding 
to $\tilde v'$ scaled by a factor  $\eta^\mez$, that is 
$\tilde w = \eta^\mez  w + \eta^{-\mez}  Q_{22}^{-1}  Q_{21 } ( A_{11} u + A_{12} v) $. 
After this change of variables, the matrix $\mA$ it conjugated   
to $\widetilde \mA$  with 
\begin{equation}
\eta^\mez \widetilde \mA = 
\begin{pmatrix}
 0& 0 & 0\\
0 & 0 &   I\\
0  &   -  Q_{22}  & 0 \\
\end{pmatrix} +  \eta^{- \mez}\begin{pmatrix}
 A^*_{11}& A_{12} & 0\\
0 & 0 & 0 \\
O (\eta^{-\mez})        &   O (\eta^{-\mez})     & A_{22}\\
\end{pmatrix}.
\end{equation}
  From (i),  the matrix $\eta^\mez \widetilde \mA$ has no eigenvelue
  on the imaginary axis, and the number of eigenvalues in 
  $\{ \Re \lambda > 0 \}$ is independent of $\eta$, and thus can be determined
 taking $\eta$ to infinity. 
 The limiting matrix has $r$ eigenvalues in 
 $\{ \Re \lambda > 0 \}$,  $r$ eigenvalues in 
 $\{ \Re \lambda < 0 \}$   and the eigenvalue $0$ with multiplicity $n$, since  
$-Q_{22}$ has its spectrum in   $\{ \Re \lambda > 0 \}$.  
The classical perturbation theory as in \cite{MaZ1} shows that 
for $\eta^{- \mez}  $ small,   $\eta^\mez \widetilde \mA$ has 
$n $ eigenvalues of order $\eta^{- \mez} $, close to the spectrum of 
$A_{11}^* $ with error $O (\eta^{-1})$.  
 Thus, for $\eta > 0$ large,  
 $\eta^\mez \widetilde \mA$ has $r + \dim \cS (A_{11}^*)$ eigenvalue 
 in $\{ \Re \lambda < 0 \}$, proving \eqref{dims}. 
 
  The proof of the Proposition is now complete. 
\end{proof}

\subsection{Finite dimensional approximations} 
To treat the infinite-dimensional case, we proceed
by finite-dimensional approximations.
Let $\underline Q = Q_{\underline M} $
and $\underline K = K_{\underline M}$ denote the operators  
$Q_U$ and $K_U$ evaluated at the equilibrium 
$\underline M = M (\underline u)$, so that 
\begin{equation}
\label{prf1}
\underline A = (\underline A)^*, \quad 
 \underline Q = (\underline Q)^* = 
 \begin{pmatrix} 0 &0 \\ 0 & \underline Q_{22} \end{pmatrix},  
\end{equation} 
with 
\begin{equation}
\label{prf3}
 \underline Q_{22}  \ge c \Id, \quad  c > 0. 
 \end{equation}
Moreover, the Kawashima multiplier has the form 
\begin{equation}
\label{prf4}
 \underline K =  - (\underline K)^* =  \theta 
 \begin{pmatrix} \underline K_{11}  & \underline K_{12} 
  \\ \underline K_{21}  & 0  \end{pmatrix}.   
\end{equation} 
Thanks to \eqref{prf3} the condition \eqref{pr5} is satisfied for 
$\theta$ small enough as soon as 
\begin{equation}
\label{prf5}
\Re \big( \underline K_{11} \underline A_{11}  +  
\underline K_{12} \underline A_{21} \big)    \ge c \Id , \quad  c > 0. 
\end{equation}

Consider an increasing  sequence of finite dimensional subspaces 
\begin{equation}
\label{pra1}
\VV_r \subset \VV_{r+1} , \qquad  \cup \VV_r \ \mathrm{dense \ in } 
\ \VV. 
\end{equation} 
Similarly, let $\HH_r = \UU \oplus \VV_r$. 
Let $\Pi_r(U) $ denote the orthogonal  projector onto
$\HH_r$, so that
\be\label{prf6}
\Pi_r=\Pi_r^*,
\ee
 and define
\begin{equation}
\label{prf7}
A_r (U) = \Pi_r( U) A \Pi_r(U), 
\qquad  Q_r(U) =  \Pi_r (U) Q_U \Pi_r(U). 
\end{equation}

\begin{lem}
\label{lemappf}
(i)  $A_r(U) $ is symmetric.

(ii)   $ Q_r(U)$ is uniformly definite negative 
on $\{ 0 \} \oplus \VV_r$. 

(iii)  With $K_r :  \Pi_r( U) K_U  \Pi_r(U)$, the Kawashima condition 
\begin{equation}
\label{prf9} 
\Re   K_r  A_r  -  Q_r   \ge \gamma \Id  
\end{equation}
is uniformly satisfied for $U$ in a neighborhood of $\underline M$
and $r$ large enough. 
\end{lem}

\begin{proof}
(i) follows by \eqref{prf6} and
symmetry of $A$ on $\HH$.
 
 On $\{ 0 \} \oplus \VV_r$, $ Q_r(U)$ is a perturbation of 
 $\pi_r \underline Q_{22} \pi_r $ where $\pi_r $ is the (usual) 
orthogonal projection onto $\VV_r$, with
 $$
 \underline \Pi_r \begin{pmatrix} u \\ v \end{pmatrix} = 
  \begin{pmatrix} u \\ \pi_r v \end{pmatrix}. 
 $$
 
 To prove \eqref{prf9}, is is sufficient to prove the property at $U = \underline M$. 
 Restricting     to $\VV_r$ by $\pi_r$,  one has 
$$
   \underline A_r  =      
 \begin{pmatrix} \underline A_{11}  & \underline A_{12} \pi_r  
  \\ \pi_r \underline A_{21}  & \pi_r \underline A_{22} \pi_r   \end{pmatrix}, 
  \qquad 
  \underline K_r  =    \theta 
 \begin{pmatrix} \underline K_{11}  & \underline K_{12} \pi_r  
  \\ \pi_r \underline K_{21}  & 0  \end{pmatrix}.   
$$
Note that  
$$ 
  \underline K_{11} \underline A_{11}  +  
\underline K_{12} \pi_r \underline A_{21}     
$$
is an $n$-dimensional perturbation of 
$ \underline K_{11} \underline A_{11}  +  
\underline K_{12}  \underline A_{21} $  whose real part is definite positive. 
Therefore, \emph{ for $r$ large enough},  
$$
\Re\big(   \underline K_{r} \underline A_r\big)_{11}=
\Re\big(   \underline K_{11} \underline A_{11}  +  
\underline K_{12} \pi_r \underline A_{21}    \big) \ge c \Id
$$
with $c$ independent of $r$.  Since 
$$
\pi_r \underline Q_{22} \pi_r \ge  c_1 \Id  \qquad \mathrm{on} \ \VV_r
$$ 
uniformly in $r$, and since the other blocks of 
$K_r A_r$  are uniformly $O(\theta)$, the condition \eqref{prf9} 
is satisfied for $r$ large enough and $\theta$ sufficiently small. 
\end{proof}

\begin{cor}
On  $\HH_r$  the 
equation 
\begin{equation}
A_r (\bar u_{NS}) \D_x U -   Q_r (\bar u_{NS}) U   = \begin{pmatrix}  
f'\\ g \end{pmatrix}  , \qquad \ell \cdot u(0) = 0 
\end{equation} 
is well posed, and there are uniform estimates in $r$, for
$r$ sufficiently large. 
\end{cor}

\begin{cor}
On  $\HH$  the 
equation 
\begin{equation}
A  \D_x U -   Q (\bar u_{NS}) U   = \begin{pmatrix}  
f'\\ g \end{pmatrix}  , \qquad \ell \cdot u(0) = 0 
\end{equation} 
is well posed.
\end{cor}

\subsection{Proof of Proposition \ref{invprop}}
Let $(\csLe)^\dagger$ denote the inverse operator of 
$\csLe$ defined by \eqref{modeq}, for   sufficiently small $\eta>0$. 
The uniform bound \eqref{unifbds83}, and weak compactness of the
unit ball in $H^2$,  for $F \in \sS$, we obtain existence of a weak solution $U \in H^2$ of
\begin{equation} 
\label{eq820}
\csL U  = F:= \begin{pmatrix}f\\g\end{pmatrix}, \qquad \ell_\eps \cdot u(0 ) = 0,  
\end{equation}
 along some weakly
convergent subsequence.
Proposition~\ref{prop72} implies uniqueness in $H^2$ for this problem,  
therefore the full family converges, giving sense to the definition 
\begin{equation}
\csLd = \lim_{\eta \to 0}  (\csLe)^\dagger
\end{equation}
acting from $\sS$ to $H^2$. 

 For $F \in \cS_{\eps \delta}$, the uniform bounds \eqref{unifbds83} imply that 
 the limit $  \csLd U \in H^s_{\eps, \delta}$ and satisfies same estimate. 
 By density, the operator $\csLd$  extends to 
 $f \in H^{s+1}_{\eps, \delta}$ and $g \in H^{1}_{\eps, \delta}$, with  
 $\csLd F \in H^s_{\eps, \delta}$.

 The sharp  bound \eqref{invbdH2} and \eqref{invbdHs} 
  now follow  immediately from Propositions~\ref{prop72} and 
  \ref{prop73}. The proof of Proposition~\ref{invprop} is now complete.


\section{Other norms}\label{othernorms}

We now briefly discuss the modifications needed to 
obtain the full result of Theorem \ref{mainthm}.

\subsection{Pointwise velocity estimates}


\begin{proof}[Proof of Theorem \ref{mainthm} ($\HH^{1/2}$)]
To obtain pointwise bounds with respect to velocity,
we carry out the same argument as in the proof
of Proposition \ref{main}, substituting
in place of the $L^2$ norm $|\cdot|$ in $\xi$, the
weighted $H^s$ (Sobolev) norm
$$
|f|_{s}:=
\sum_{k=0}^s C^{-k} |\D_\xi^k f|_{^2},
$$
$C>0$ sufficiently large, similarly as we did for the $x$-variable
in order to get pointwise bounds in $x$.

We have only to observe that differentiating the linearized equations in $\xi$
gives the same principal part applied to the $\xi$-derivative of $U$,
plus commuator terms.
Since commutator terms, both for the linearized collision operator
$L$ and the transport operator $A$ are of one lower derivative in
$\xi$ and also one lower factor in $\la \xi \ra$ (straightforward
computation differentiating $|\xi-\xi'|$, $\xi_1$, respectively)
for the hard-sphere case, we easily find that commutator terms
are absorbable for $C>0$ sufficiently large by lower order estimates
already carried out.


Thus, we obtain all the same estimates as before and the argument closes
to give the same result in the stronger norm $|\cdot|_s$.
(Note: a detail is to observe that 
truncation errors of the approximate solution are of the same
order in the Sobolev norm, which follows
by the corresponding property of the Maxwellian.) 
Applying the Sobolev embedding estimate in $\xi$, we obtain
\eqref{finalbds} for $\eta=1/2$, which evidently implies the
same estimate for $\eta\ge 1/2$.
\end{proof}

\subsection{Higher weights}\label{generalwt}

\begin{proof}[Proof of Theorem \ref{mainthm} ($\HH^s$)]
To extend our results from $\HH^{\frac{1}{2}}$ to $\HH^s$,
we use a simple bootstrap argument together with the key
observation that the $\HH^s$ norm of $\PP_\UU f$ is controlled
(by equivalence of finite-dimensional norms) by the $\HH^{\frac{1}{2}}$
estimates already obtained.
Namely, starting similarly as in \eqref{lineareq} with the equation
$ A\partial_x -L_a U=F$, 
$ \PP_\UU F= f$, $\PP_\VV F=g$,
we find, taking the 
$\HH^s$-inner product of $U$ against this equation and
applying the result of Proposition \ref{propcoercs} and
recalling that $A$ is formally self-adjoint in $\HH^s$,
we obtain the estimate
\begin{equation}\label{friedeq}
\| \PP_{\VV} U \|_{L^2}  \le C \big(     
\|  f  \|_{L^2}   +    \| g  \|_{L^2}   
+ \| \PP_{\UU} U \|_{L^2} \big). 
\end{equation}
Differentiating the equations $k$ times and taking the inner
product with $\partial_x^k U$, we find, similarly,
the higher-derivative estimate
\begin{equation}\label{friedeqk}
\| \PP_{\VV} U \|_{H^k_{\eps,\delta}}  \le C \big(     
\|  f  \|_{H^k_{\eps,\delta}}   +    \| g  \|_{H^k_{\eps,\delta}}   
+ \| \PP_{\UU} U \|_{H^k_{\eps,\delta}} \big). 
\end{equation}

Specializing now to the case \eqref{neweq}, \eqref{phasecond},
and bounding the $\HH^s$ norm of $\PP_\UU U$ by a constant times
the $\HH^{\frac{1}{2}}$ bound obtained already in our previous analysis,
we recover the key bounds \eqref{invbdH2}--\eqref{invbdHs}
of Proposition \ref{invprop} in the general space $\HH^s$.
With this bound, the entire contraction mapping argument goes
through in $\HH^s$, since this relies only on boundedness estimates
on $A$, $Q$ already obtained, the estimate \eqref{redbds} (still valid in 
$\HH^s$), and the linearized estimates \eqref{invbdH2}--\eqref{invbdHs},
yielding \eqref{finalbds}(i) and (ii) as claimed,
for any $\eta >0$.

The estimate \eqref{finalbds}(iii) then follows by
Remark \ref{fastdecay} estimating decay in velocity $\xi$ of
the approximating profile $\bar f_{NS}$.
\end{proof}

\begin{rem}\label{nonsa}
\textup{
We emphasize that $L$ is {\it not} approximately self-adjoint with respect
to $\HH^s$, $s >> 1/2$, and, likewise,
 the splitting $\HH^s = \UU \oplus \VV^s$ using projectors 
 $\PP_U$ and $\PP_V$ is not orthogonal in this norm.
For this reason, we obtain term
$ \| \PP_{\UU} U \|_{H^k_{\eps,\delta}}$ in the righthand side of
\eqref{friedeqk} and not 
$\eps \| \PP_{\UU} U \|_{H^k_{\eps,\delta}}$ as in the $\HH^{1/2}$
computations.
The missing $\eps$ factor was crucial in closing the argument in $\HH^{1/2}$
and estimating $ \PP_{\UU} U $.
However, with $ \PP_{\UU} U $ already bounded it is no longer needed,
since our final estimates make no distinction between
$ \PP_{\UU} U $ and $ \PP_{\VV} U $ components; that is, the lost
$\eps$ factor is needed only to close the loop between microscopic
and macroscopic estimates, and not to
bound $\PP_{\VV}U$ in terms of $\PP_{\UU}U$.
}
\end{rem}

\section{Other potentials}\label{genpot}

Finally, we briefly indicate the changes needed to
accomodate general hard cutoff potentials.
Recall \cite{CN,GPS} that these give structure $L=-\nu(\xi)+\cK$,
where $\nu \sim \la \xi\ra^{\beta}$, $0<\beta <1$, and $\cK$ is
compact, and similarly for $Q$.
Dividing by $\nu \sim \la \xi\ra^{\beta}$ as before, we can
thus obtain $Q$, $L$ bounded, but this leaves $A$ unbounded.  
Nonetheless, a closer look shows that the Kawashima compensator
$K$ as constructed is still bounded, the key point.  
For, examining $A_{12}$, we see that it decays as a polynomial
in $\xi$ times full Maxwellian, so is clearly bounded in $\HH^s$
for $s<1$.

Since the norm of $A$ does 
not enter except through the good term $KA$, our basic micro-estimates 
therefore still survive.  
Of course, the macro-estimates, since finite-dimensional,
survive as well.  (This follows by the same estimate
that shows that $K$ as constructed is bounded; that is, one has only to
check that $A_{12}$ and $A_{21} $ entries remain bounded,
thanks to Maxwellian rate decay.)
Thus, the argument goes through as before, also for these more
general potentials.


\end{document}